\documentclass[sn-mathphys,Numbered]{sn-jnl}

\usepackage{graphicx,xcolor}%
\usepackage{multirow}%
\usepackage{amsmath,amssymb,amsfonts, upgreek}%
\usepackage{amsthm,nicefrac}%
\usepackage{mathrsfs}%
\usepackage[title]{appendix}%
\usepackage{xcolor}%
\usepackage{textcomp}%
\usepackage{manyfoot}%
\usepackage{booktabs}%
\usepackage[vlined, ruled, algo2e, linesnumbered]{algorithm2e}
\usepackage{listings}%
\usepackage{bm}
\usepackage{enumitem}
\usepackage{tabularx}

\usepackage{bm}
\usepackage{bbm}
\usepackage{mathtools}   
\usepackage{accents}
\usepackage{graphicx} 
\usepackage{comment}
\usepackage{xcolor}
\usepackage{multirow}
\usepackage{hyperref}
\usepackage{cleveref}

\DeclareMathOperator*{\argmin}{arg\,min}

\newcommand{\R}{\mathbb R}
\newcommand{\N}{\mathbb N}
\makeatother
\def\theequation{\thesection.\arabic{equation}}
\newcommand{\bx}{{\bm x}}

\newcommand{\yd}{{y^{\delta}}}

\newcommand{\Q}{{\mathscr Q}}

\newcommand{\Csp}{{\mathscr C}}
\newcommand{\Qad}{{\mathscr Q_\mathsf{ad}}}
\newcommand{\qad}{{Q_{\mathsf{ad},h}}}
\newcommand{\qadK}{{Q_{\mathsf{ad},h}^K}}

\newcommand{\qadrK}{{Q_{\mathsf{ad},r}^K}}

\newcommand{\qadriK}{{Q^{K, \iteridx}_{\mathsf{ad},r}}}

\newcommand{\Sol}{\mathcal S}
\newcommand{\C}{{\mathcal C}}
\newcommand{\F}{{\mathcal F}}
\DeclareMathOperator*{\essinf}{ess\,inf}

\newcommand{\ulin}{\tilde{u}}
\newcommand{\plin}{\tilde{p}}
\newcommand{\iteridx}{{(i)}}
\newcommand{\errest}{\mathcal{R}^\iteridx_{J_r}}
\newcommand{\Qridx}{(Q^{\iteridx}_r)^K}

\newcommand{\qtrial}{q^\iteridx_{\mbox{\tiny{trial}},r}}
\newcommand{\epsPOD}{\epsilon_\text{POD}}

\definecolor{DarkRed}{rgb}{0.5 0 0}
\definecolor{DarkBlue}{rgb}{0 0 0.5}
\definecolor{DarkGreen}{rgb}{0 0.5 0}
\definecolor{DarkOrange}{rgb}{0.83 0.33 0}
\definecolor{DarkMagenta}{rgb}{0.83 0 0.67}

\newtheoremstyle{Stil}{}{}{\itshape}{}{\bfseries}{.\\}{1ex}{}
\newtheoremstyle{Stil2}{}{}{\rmfamily}{}{\bfseries}{.}{1ex}{}
\theoremstyle{Stil}
\newtheorem{theorem}{Theorem}[section]
\newtheorem{proposition}[theorem]{Proposition}

\newtheorem{lemma}[theorem]{Lemma}

\theoremstyle{plain}
\newtheorem{remark}[theorem]{Remark}
\newtheorem{assumption}[theorem]{Assumption}

\theoremstyle{Stil2}
\newtheorem{example}[theorem]{Example}

\crefname{assumption}{\textup{Assumption}}{\textup{Assumptions}}
\crefname{lemma}{\textup{Lemma}}{\textup{Lemmas}}
\crefname{theorem}{\textup{Theorem}}{\textup{Theorems}}
\crefname{remark}{\textup{Remark}}{\textup{Remarks}}
\crefname{example}{\textup{Example}}{\textup{Examples}}
\crefname{corollary}{\textup{Corollary}}{\textup{Corollaries}}
\crefname{subsection}{\textup{Section}}{\textup{Subsections}}
\crefname{section}{\textup{Section}}{\textup{Sections}}
\crefname{figure}{\textup{Figure}}{\textup{Figures}}
\crefname{table}{\textup{Table}}{\textup{Tables}}

\begin{document}

\title[Article Title]{Adaptive Reduced Basis Trust Region Methods for Parabolic Inverse Problems}


\author[1]{\fnm{Michael} \sur{Kartmann}}\email{michael.kartmann@uni-konstanz.de}

\author[2]{\fnm{Benedikt} \sur{Klein}}\email{benedikt.klein@uni-muenster.de}

\author[2]{\fnm{Mario} \sur{Ohlberger}}\email{mario.ohlberger@uni-muenster.de}%

\author[3]{\fnm{Thomas} \sur{Schuster}}\email{thomas.schuster@num.uni-sb.de}%

\author[1]{\fnm{Stefan} \sur{Volkwein}}\email{stefan.volkwein@uni-konstanz.de}

\affil[1]{\orgdiv{Department of Mathematics and Statistics}, \orgname{University of Konstanz}, \orgaddress{\street{Universit\"atsstr. 10}, \city{Konstanz}, \postcode{78464},  \country{Germany}}}

\affil[2]{\orgdiv{Applied Mathematics: Institute for Analysis and Numerics}, \orgname{University of M\"unster}, \orgaddress{\street{Einsteinstr. 62}, \city{M\"unster}, \postcode{48149},  \country{Germany}}}

\affil[3]{\orgdiv{Department of Mathematics}, \orgname{Saarland University}, \orgaddress{\street{Campus}, \city{Saarbrücken}, \postcode{66123},  \country{Germany}}}


\date{\today}

\abstract{
   We consider nonlinear inverse problems arising in the context of parameter identification for parabolic partial differential equations (PDEs). For stable reconstructions, {regularization methods such as the iteratively regularized Gauss–Newton method} (IRGNM) are commonly used, but their application is computationally demanding due to the high-dimensional nature of PDE discretizations. To address this bottleneck, we propose a reduced-order modeling approach that accelerates both the state and adjoint evaluations required for derivative-based optimization. Our method builds on the recent contribution [Kartmann et al. Adaptive reduced basis trust region methods for parameter identification problems.
   Comput. Sci. Eng. 1, 3 (2024)] for elliptic forward operators and constructs the reduced forward operator adaptively in an online fashion, combining both parameter and state space reduction. To ensure reliability, we embed the IRGNM iteration within an adaptive, error-aware trust-region framework that certifies the accuracy of the reduced-order approximations. We demonstrate the effectiveness of the proposed approach through numerical results for both time-dependent and time-independent parameter identification problems in dynamic reaction-diffusion systems. The implementation is made available for reproducibility and further use.}
\date{\today}




\maketitle

\section{Introduction}
\label{sec:introduction}
\noindent
An important type of dynamic inverse problems (DIPs) are parameter identification problems for time-dependent and spatially distributed parameter fields in parabolic partial differential equations (PDEs) from noise-contaminated data, see, e.g., \cite{KaltenbacherMain, diss_kirchner, kaltenbacher2021parameter, kaltenbacher2021time}. 
Such problems are a challenging task in many scientific and applied disciplines, such as structural engineering. 
A prominent application arises in Structural Health Monitoring (SHM)~\cite{binder2015defect, klein2021sequential, seydel2017linearization, seydel2017identifying, lechleiter2017identifying}, where the goal is to detect structural defects by reconstructing {dynamic loads or the stored energy function from time-dependent sensor data at the structure's surface. 
Such inverse problems (cf.~\cite{isakov2017inverse, Kavian+2003+125+162}), referred to as \textit{dynamic parameter identification problems}, are inherently {under-determined} and ill-posed in the sense that the parameters to be reconstructed do not depend continuously on the measurements~\cite{BakKok04, HaNeSc95, KaltenbacherNeubauerScherzer+2008, Kirs96,burger2024ill}. To obtain meaningful and stable solutions, regularization techniques are indispensable and have been the subject of extensive research over the past decades~\cite{TikArs77, Groe84, Kirs96}.

Numerically solving these inverse problems typically requires a large number of evaluations of the underlying PDE (forward problem), as well as potentially other equations (e.g., adjoint problems) to be solved. In this work, we focus on the iteratively regularized Gauß--Newton method (IRGNM)~\cite{MR1185952, MR1872920, MR2337589, SBH12, KP18, KKV14, KaltenbacherMain} and aim to develop a variant with improved computational efficiency by incorporating (local) reduced basis (RB) models \cite{Haasdonk2017Chapter2R}, which are well-suited for the method's multi-query nature. To enable efficient model updates and error control, this approach is embedded into an error-aware trust-region framework~\cite{YueMeerbergen2013} that has recently been advanced both theoretically and numerically~\cite{BBMV24, Keil2021nonco-54293, keil2024relaxed}. While several strategies have been proposed to enhance the efficiency of the IRGNM~\cite{GhattasWillcox2021, 2_69305, Wald2018, RBL}, our approach is particularly suited to problems involving parameter fields with high spatial resolution. In the following, we present a detailed overview of the techniques employed in this framework.

{
\textbf{Dynamic inverse problems.}
The computation of parameters in (systems of) parabolic or hyperbolic PDEs represents an important class of dynamic inverse problems. Such problems usually are inherently nonlinear (even if the PDE is linear), highly under-determined and (very often) severely ill-posed. Regularization methods demand for many evaluations of the underlying PDE, as well as of the Fr\'{e}chet derivative of the forward operator and its adjoint. Moreover, the (weak) solutions of such PDEs live in appropriate Sobolev-Bochner spaces and show different regularity in time and space. This is why standard analytical frameworks are not appropriate. In \cite{burger2024ill} the authors developed first concepts taking these phenomena into account. They distinguish between temporally (pointwise) and uniformly ill-posedness of dynamic inverse problems as well as between temporal (pointwise) and uniform regularization methods. Here, the pointwise view comes from considering the inverse problem for a fixed time instance as in variational tracking. These are indeed different concepts for ill-posedness and regularization, which is necessary for appropriate solution concepts for dynamic inverse problems, since time is not seen merely as an additional dimension of the problem. A collection of applications for dynamic inverse problems is presented in \cite{kaltenbacher2021time}.
}

\textbf{Model order reduction for parameter identification.}  
Model Order Reduction (MOR) techniques have gained significant attention in recent years as efficient tools for solving parameterized partial differential equations (pPDEs)~\cite{bennerohlberger,MR3701994}. These methods aim to reduce computational complexity by replacing high-dimensional full-order models (FOMs) with low-dimensional surrogate models. A substantial body of literature discussing the application of MOR across various problem domains exists~\cite{KTV13, GK2011, NRMQ2013, MR4628189, MR4902803}.
A particularly prominent class of MOR techniques involves approximating the solution manifold of pPDEs using low-dimensional linear subspaces spanned by a reduced basis. In this context, we refer specifically to the POD-Greedy method~\cite{Haasdonk2008, Haasdonk13} and the Hierarchical approximate POD (HaPOD)~\cite{himpe2018hierarchical}, both of which identify low-dimensional subspaces that capture the dominant behavior from a (potentially large) set of solution snapshots. However, if the parameter space gets high-dimensional, the RB reduction becomes infeasible, due to the high computational effort in constructing the reduced-order model (ROM). This motivates a simultaneous reduction of both the state and parameter spaces, which can be naturally achieved within the reduced basis framework and integrated seamlessly~\cite{kartmann_adaptive_2023}.
 We also refer to \cite{Liebermanetal2016,HimpeOhlberger2014,HimpeOhlberger2015} for alternative approaches of combined state and parameter space reduction.

\textbf{Error-aware trust region methods.}  
To address ill-posed inverse problems using the IRGNM, they are reformulated as optimization tasks that aim to identify a parameter minimizing a regularized discrepancy functional with respect to the observed data. Trust-Region (TR) methods form a class of optimization algorithms that iteratively solve subproblems defined by locally accurate surrogate models. Rather than constructing global reduced basis spaces for the parameter and state, TR approaches concentrate optimization efforts within a localized region, thereby avoiding the burden of building global surrogates.
TR optimization techniques have been effectively combined with model order reduction to leverage locally accurate surrogate models \cite{AFS00,BC08,AB13,AGH18}. If the trust region metric is defined by an error estimator, the trust region method is called error-aware. This integration ensures global convergence while significantly reducing computational costs~\cite{YueMeerbergen2013, RoggTV17}. Further theoretical developments and numerical analyses of such methods are presented in~\cite{Keil2021nonco-54293, BBMV24,keil2024relaxed,QGVW17,UEHH25}.

\textbf{Main results.}
This work extends the approach introduced in \cite{kartmann_adaptive_2023} to inverse problems governed by parabolic PDEs with time-dependent, spatially distributed parameter fields. We demonstrate the potential of combined parameter and state space reduction using a reduced basis approach within an error-aware trust-region optimization framework, enabling an efficient and certified numerical procedure for solving such inverse problems.

A central contribution of this work is the integration of truncated Proper Orthogonal Decomposition (POD) into the enrichment process of the reduced basis spaces. This method guarantees that the reduced bases are extended only by the minimal number of vectors necessary to meet a specified error tolerance, thereby preventing the formation of excessively large reduced spaces that could hinder computational efficiency. Consequently, the approach involves the use of inexact gradients, which require additional safeguards within the optimization algorithm to ensure convergence and avoid stagnation.

The proposed algorithm is implemented in a modular and extensible software framework~\cite{source_code}, built on top the model order reduction library \texttt{pyMOR} \cite{doi:10.1137/15M1026614}. Numerical experiments demonstrate that the approach achieves substantial reductions in computational time for reconstructing time-dependent and time-independent reaction and diffusion fields.

\textbf{Organization of the article.} In Section~\ref{sec:param_ident_prob}, the IRGNM is introduced in an abstract formulation for a generic parameter identification problem and its application to problems defined by parabolic PDEs is discussed. Section~\ref{sec:num_models} outlines the discretization of the governing equations and, based on this, a reformulation into a setting with reduced parameter and state spaces. In Section~\ref{sec:TR-IRGNM}, this reduced framework is combined with the IRGNM into an error-aware trust-region algorithm, employing adaptively enriched reduced basis spaces generated from previously computed snapshots via POD. Finally, numerical experiments in Section~\ref{sec:num_exper} demonstrate the capabilities and limitations of this algorithm.

%

\section{The parabolic parameter identification problem}
\label{sec:param_ident_prob}
In this section, we introduce in Subsection~\ref{subsec:IRGNM} the iteratively regularized Gauss--Newton method in an abstract setting, and present the bilinear parabolic parameter identification in Subsections~\ref{ssec:parabolic_par_id} and \ref{ssec:realization_IRGNM}.
\subsection{The iteratively-regularized Gauss-Newton method (IRGNM)}
\label{subsec:IRGNM}
Let $\Csp$ and $\Q$ be real Hilbert spaces and $\Qad\subset \Q$ be a non-empty, closed, and convex set.
We consider the forward operator $\F:\Qad\to\Csp$, which is assumed to be twice continuously Frechet differentiable. We are interested in the inverse problem of identifying a parameter $q^\mathsf e\in\Qad$ such that for given exact data $y^\mathsf e\in \text{range}(\F)$ the following equation holds:
\begin{align}
	\label{eq: IP}
	\F(q^\mathsf e)=y^\mathsf e\quad\text{in }\Csp.
\end{align}
Throughout our work, we assume that the pair $(y^\mathsf e, q^\mathsf e)$ is unique and that instead of the exact data $y^\mathsf e$ only a noisy measurement $\yd\in \Csp$ is available with
\begin{align*}
	{\|y^\mathsf e-\yd\|}_\Csp\leq \delta,
\end{align*}
where the noise level $\delta>0$ is assumed to be known.  This poses a challenge in reconstructing $q^\mathsf e$, since we assume \eqref{eq: IP} to be ill-posed in the sense that $\F$ is not continuously invertible. Therefore direct inversion of the noise-contaminated data $\yd$ would not lead to a meaningful solution, but instead, regularization methods need to be applied. In this manner, we are interested in minimizing the \emph{discrepancy}
\begin{equation}
	\label{eq:minJ}
	\tag{$\mathbf{IP}$}
	\min J(q)\coloneqq \frac{1}{2} \,{\|\F(q)-\yd\|}_\Csp^2 \quad \text{subject to (s.t.) }\quad q\in\Qad,
\end{equation}
by the \emph{iteratively regularized Gauss-Newton method} (IRGNM) \cite{KKV14}. Since in general ${\yd\notin \mbox{range}(\F)}$, \eqref{eq:minJ} does not need to possess a solution, but it is possible to approximate a reconstruction based on early-stopping as follows.  Let $q^\iteridx\in \Qad$ be an iterate sufficiently close to the solution $q^\mathsf e$ and let $i\in \N$ denote the iteration counter.
The IRGNM update scheme results from linearizing $\F$ at $q^\iteridx$ and minimizing the Tikhonov functional of the linearization
\begin{equation}
	\label{eq:IRGNMscheme_minimize}
	\tag{$\widetilde{\textbf{IP}}_\alpha$}
	d^\iteridx \coloneqq \argmin\tilde{J}(d; q^\iteridx, \alpha^\iteridx)\quad \text{s.t.} \quad q^\iteridx + d \in\Qad,
\end{equation}
where the linearized objective is given as
\begin{equation*}
    \tilde{J}(d; q^\iteridx, \alpha^\iteridx) \coloneqq  \frac{1}{2}\,{\|\F(q^\iteridx)+\F'(q^\iteridx)d-\yd\|}_\Csp^2 + \frac{\alpha^\iteridx}{2}\,{\|q^\iteridx + d -q_\circ \|}_\Q^2,
\end{equation*}
with $q_\circ\in\Q$, the center of regularization, and $\alpha^\iteridx >0$, a Tikhonov regularization parameter. The regularization parameter $\alpha^\iteridx$ is chosen a posteriori. We accept the iterate $q^{(i+1)}\coloneqq q^\iteridx + d^\iteridx$ if $\alpha^\iteridx$ satisfies
\begin{equation}
	\label{eq: choice of alpha}
	\theta J(q^\iteridx)\leq 2\tilde{J}(d^\iteridx; q^\iteridx, 0)  =  {\|\F'(q^\iteridx)d^\iteridx+\F(q^\iteridx)-\yd\|}_\Csp^2 \leq \Theta J(q^\iteridx),
\end{equation}
where $0<\theta<\Theta<2$ holds. As stopping criteria, we use the so-called \emph{discrepancy principle}: the iteration is stopped after $i_*(\delta,\yd)$ steps provided
\begin{equation}\label{eq: discrepancy principle}
	J(q^{(i_*(\delta,\yd))})\leq \frac{1}{2}(\tau \delta)^2 \leq  J(q^\iteridx), \quad\ i=0,...,i_*(\delta,\yd).
\end{equation}
The parameter $\tau>1$ reflects that we cannot expect the discrepancy to be lower than the noise in the given data. In \cite{KKV14} a local convergence and regularization result was proven under additional assumptions on $\F$. To be precise, it was shown that the early-stopping index $i_*(\delta,\yd)$ is finite and one has $q^{(i_*(\delta,\yd)))}\to q^{\mathsf e}$ as $\delta \searrow 0$. The procedure is summarized in \Cref{algo:IRGNM}.
\begin{algorithm2e}
\DontPrintSemicolon
\caption{IRGNM}\label{algo:IRGNM}
	\KwData{Noise level $\delta$, discrepancy parameter $\tau>1$,  initial guess $q^{(0)}$,  initial regularization $\alpha^{(0)}>0$, regularization center $q_\circ$.}
	Initialize $i=0$.\;
	\While{$\|\F(q^\iteridx)-\yd\|_\Csp > \tau \delta$}{
		Solve subproblem \eqref{eq:IRGNMscheme_minimize} for $d^\iteridx$.\;
		\While{\eqref{eq: choice of alpha} is not fulfilled}{
			Increase (decrease) $\alpha^\iteridx$ if the first (second) part of \eqref{eq: choice of alpha} is not valid.\;
			Solve subproblem \eqref{eq:IRGNMscheme_minimize} for $d^\iteridx$.\;
		}
		Set $q^{(i+1)}=q^\iteridx + d^\iteridx$, $i=i+1$.\;
	}
\end{algorithm2e}
%
\subsection{The parabolic parameter identification problem}\label{ssec:parabolic_par_id}
%
Let $V$ and $H$ be real Hilbert spaces, with $V$ densely embedded into $H$. Let $T > 0$ and define $\Q \subset L^2(0,T; Q)$, $\Csp\coloneqq L^2(0,T;C)$ the parameter and output space, respectively, where $C$ and $Q$ are real Hilbert spaces. For $q \in \Qad \subset \Q$, let $u=\Sol(q) \in W(0,T)\coloneqq L^2(0,T; V) \cap H^1(0,T; V')$, be the \emph{primal state}, defined as the solution to the parabolic PDE
\begin{align}
    \label{eq:weak_problem}
    \langle \partial_t u(t)+ A(q(t))u(t),v \rangle_{V',V} =  \langle l(t),v \rangle_{V',V}   , \quad u(0) = 0\text{ in }H
\end{align}
for all $v \in V$ and f.a.a. (for almost all) $t\in (0,T)$. The solution operator to \eqref{eq:weak_problem} $\Sol:\Qad\to W(0,T)$ is well-defined for $l(t) \in V'$ f.a.a. $t\in (0,T)$ under the following assumption.

\begin{assumption}
    \label{ass:bilinear_form}
    The operator $A(q(t)) \in \mathcal{L}(V, V')$ is coercive and continuous for all $q\in\Qad$, f.a.a. $t\in (0,T)$, i.e., there exists a positive $a\in \R$ such that for all $q\in \Qad$ there exists $a_{q(t)}\geq a$
    \begin{equation}
        \langle A(q(t))v,v \rangle_{V',V}\geq a_{q(t)}\|v\|^2_V \quad \mbox{ for } t\in (0,T)
    \end{equation}
    for all $v \in V$. We further assume the map $q(t) \mapsto A(q(t))$ to be affine.
\end{assumption}
The forward operator in \eqref{eq: IP} takes the form $\F:\Qad\to\Csp\quad\text{with }\F=\mathcal C\circ\mathcal S$, where $\Sol$ specified above and $\mathcal{C} \in \mathcal L(L^2(0,T; V),\Csp)$ is a linear and bounded observation operator. In the following, we assume for simplicity that $\mathcal C $ is independent of $t$. Hence, we can understand the mapping $\mathcal{C}$ also as an mapping $\C\in \mathcal L(V,C)$ via $\C u(t)=(\C u)(t)$ for all $u\in L^2(0,T;V)$ and $t\in (0,T)$.
\begin{example}
    \label{ex:default_example}
    We consider the two examples, introduced in \cite{kartmann_adaptive_2023}, extending them to the parabolic setting. Let $\Omega \subset \R^d$ be a convex domain and $V\coloneqq H^1_0(\Omega)$, $H\coloneqq L^2(\Omega)$, $C\coloneqq H$. The operator $\C$ is chosen to be the compact, injective embedding from $W(0,T)$ into $L^2(0,T;H)$.
    \begin{enumerate}[label=(\roman*.)]
        \item \textit{Reconstruction of a reaction field:} Let $q_\text{a} \in L^\infty(\Omega)$ be such that $\essinf_{\Omega }q_\text{a}>0$. We choose $Q = L^2(\Omega)$ and $\Q = L^2(0,T; L^2(\Omega))$. The set of admissible parameter functions is given as  
        \begin{equation}
            \label{eq:Q_ad}
            \Q_\text{ad}\coloneqq\lbrace q \in \Q\,|\,q_\text{a} \leq q \text{ a.e. in } (0,T) \times \Omega \rbrace.
        \end{equation}
        For $q \in \Q_\text{ad}$ and $f\in L^2(0,T; H)$ consider 
        \begin{align}
            \label{eq:reaction_problem}
            \begin{aligned}
                \partial_t u(t,\bx)-\Delta u(t, \bx)+q(t,\bx) u(t,\bx)&=f(t,\bx)&&\quad \text{for }(t,\bx)\in (0,T)\times\Omega, \\
                u(t,\bx)&=0&&\quad\text{for }(t,\bx) \in (0,T)\times\partial\Omega,\\
                u(0,\bx)&=0&&\quad\text{for }\bx\in \Omega.
            \end{aligned}
        \end{align}
        The weak formulation of \eqref{eq:reaction_problem} can be expressed
        in the form \eqref{eq:weak_problem} by setting f.a.a. $t \in (0,T)$
        \begin{equation*}
            \langle A(q(t))v,w \rangle_{V',V}\coloneqq\int_\Omega \nabla v \cdot \nabla w + q(t, \bx) v w\,\text{d}\bx\quad\text{for } v,w \in V
        \end{equation*}    
        and 
        \begin{equation*}
            \langle l(t),v \rangle_{V',V}\coloneqq \int_\Omega f(t)v\,\text{d}\bx\quad \text{for } v \in V.
        \end{equation*}
        \item \textit{Reconstruction of a diffusion field:} Let $\Q_\text{ad}$ be given as in \eqref{eq:Q_ad}. We define the parameter space by $Q = H^2(\Omega)$ and ${\Q\coloneqq L^2(0,T; H^2(\Omega)) \hookrightarrow L^2(0,T; C(\bar{\Omega}))}$. For given $q \in \Q_\text{ad}$ and $f\in L^2(0,T; H)$ consider 
        \begin{align}
            \label{eq:diffusion_problem}
            \begin{aligned}
                \partial_t u(t, \bx)- \nabla(q(t,\bx)\cdot\nabla u(t,\bx)) &=f(t,\bx)&&\text{for }(t,\bx)\in (0,T) \times \Omega,\\
                u(t, \bx)&=0&&\text{for }(t,\bx)\in(0,T)\times \partial\Omega,\\
                u(0,\bx)&=0&&\text{for }\bx\in \Omega.
            \end{aligned}
        \end{align}
        For 
        \begin{equation*}
            \langle A(q(t))v,w \rangle_{V',V}\coloneqq\int_\Omega q(t,\bx) \nabla v \cdot \nabla w \,\text{d}\bx\quad\text{for }v,w\in V
        \end{equation*}   
        and 
        \begin{equation*}
            \langle l(t),v \rangle_{V',V}\coloneqq \int_\Omega f(t)v\,\text{d}\bx\quad\text{for } v \in V
        \end{equation*}
        the weak formulation of \eqref{eq:diffusion_problem} is given by \eqref{eq:weak_problem}.
    \end{enumerate}
\end{example}
\begin{remark}
    \label{rem:example_time_indep}
    The Examples \ref{ex:default_example} include the situation where the parameter fields are considered to be constant in time, i.e., $q(t,\bx) = \tilde{q}(\bx)$ for all $t \in (0,T)$. In this case, the parameter fields are assumed to be elements of $L^2(\Omega)$ and $H^2(\Omega)$, respectively, with the set of admissible parameters given as
    \begin{equation*}
      {\Q}_\text{ad} := \left\lbrace q \in \Q \mid q_\text{a} \leq q \text{ a.e. in } \Omega \right\rbrace.
    \end{equation*}
\end{remark}
%

%
\subsection{Realization of the IRGNM}
\label{ssec:realization_IRGNM}
Let $A(q)'\in \mathcal{L}(V,V')$ denote the dual operator of $A(q)$ satisfying $\langle A(q)'p,u\rangle_{V',V}= \langle A(q)u,p\rangle_{V',V}$ for $u,p\in V$ and $q\in Q$. 
The adjoint equation w.r.t. the objective in \eqref{eq:minJ} is given for $q\in \Qad$ as
\begin{align}\label{eq:adjoint_eq}
     \langle -\partial_t p(t)+A(q(t))'p(t), v \rangle_{V', V} 
    = - \langle \C u(t)-y^\delta(t), \C v \rangle_C, \quad p(T) = 0,
\end{align}
for all $v \in V$ and f.a.a. $t\in(0,T)$. Under Assumption \ref{ass:bilinear_form}, a unique solution $p \in W(0,T)$ exists, called the adjoint state. If $u$ and $p$ are the state and the adjoint state respectively corresponding to $q$, the gradient at $q$ in the direction $e \in \Q$ is given as
\begin{equation}
    \langle \nabla J(q), e \rangle_\Q =  
    \int_{0}^T\langle A(e(t))u(t), p(t) \rangle_{V',V}\,\text{d}t.
\end{equation}
Let ${u^\iteridx = \mathcal{S}(q^\iteridx) \in W(0,T)}$ be the state corresponding to the current iterate ${q^\iteridx \in \Qad}$. The iterates are updated according to
\begin{equation}
    q^{(i+1)} = q^\iteridx + d^\iteridx \in \Qad,
\end{equation}
where $d^\iteridx$ solves the strictly convex optimization problem \eqref{eq:IRGNMscheme_minimize} for some regularization parameter $\alpha^\iteridx > 0$. The update $d^\iteridx$ is optimal if and only if, cf.~\cite{nocedal2006numerical},
\begin{equation*}
    d^\iteridx = P_{\Qad} \left[q^\iteridx + d^\iteridx - \nabla_d \tilde{J}(d^\iteridx; q^\iteridx, \alpha^\iteridx) \right]- q^\iteridx ,
\end{equation*}
where $P_{\Qad}:\Q \rightarrow\Qad$ denotes the projection onto the non-empty, closed convex set $\Qad \subset \Q$, defined by
\begin{equation}
    \label{eq:projection_operator}
    P_{\Qad}(x)\coloneqq\argmin \|x - y\|_\Q \text{ s.t. } y \in \Qad.
\end{equation}
The gradient of the linearized objective in direction $e \in \Q$, for $u = \mathcal{S}(q)$, is given by
\begin{equation*}
    \label{eq:lin_grad}
    \langle \nabla_d \tilde{J}(d; q, \alpha), e \rangle_{\Q} = \int_0^T \langle A(e(t))u(t), \plin(t) \rangle_{V', V} + \langle \alpha (q(t) + d(t) - q_\circ(t)), e(t) \rangle_{\Q} \,\text{d}t,
\end{equation*}
where the linearized adjoint state $\plin \in W(0,T)$ is determined by solving the system
\begin{align*}
	0 &= \langle \partial_t \ulin(t) + A(q(t))\ulin(t), v \rangle_{V', V} + \langle A(d(t))u(t), v \rangle_{V', V}, \\
	0 &= \langle -\partial_t \plin(t) + A(q(t))'\plin(t), v \rangle_{V', V} + \langle \C(u(t) + \ulin(t) - y^\delta(t)), \C v \rangle_C,
\end{align*}
for all $v \in V$, $q \in \Q$, and f.a.a. $t \in (0, T)$, with initial and terminal conditions $\ulin(0) = 0$ and $\plin(T) = 0$. We refer to $\ulin \in W(0,T)$ as the linearized state.
\section{Numerical models}
\label{sec:num_models}
In this section, we introduce the finite element model corresponding to the system derived in Section~\ref{sec:param_ident_prob}, and outline the IRGNM procedure applied to it. Subsequently, in Subsection~\ref{sec:modelreduction}, we define the combined parameter- and state-reduced basis model associated with this full-order model (FOM).
\subsection{Finite Element model}
\label{ssec:discretization}

Let defined finite-dimensional linear spaces $V_h = \mathrm{span}(\varphi_1, \dots, \varphi_{N_V}) \subset V$, ${Q_h = \mathrm{span}(q_1, \dots, q_{N_Q}) \subset Q}$, and $C_h = \mathrm{span}(c_1, \dots, c_{N_C}) \subset C$ with dimensions $N_V$, $N_Q$, and $N_C\in \N$, respectively. We denote the corresponding space-discrete objects as\newpage
\begin{align*}
    \bm{M}_Q &\coloneqq \left(\langle q_i,q_j \rangle_{Q}\right)_{i,j \in \{1, \dots, N_Q\}},&  \bm{M}_V &\coloneqq \left(\langle \varphi_i,\varphi_j \rangle_{V}\right)_{i,j \in \{1, \dots, N_V\}},\\
    \bm{M}_H &\coloneqq \left(\langle \varphi_i,\varphi_j \rangle_{H}\right)_{i,j \in \{1, \dots, N_V\}},& \bm{M}_C &\coloneqq \left(\langle c_i,c_j \rangle_{C}\right)_{i,j \in \{1, \dots, N_C\}},\\
    \bm{L}_h(t)&\coloneqq\left(\langle l(t),\varphi_j\rangle_{V',V}\right)_{j\in \{1,\dots,N_V\}}&\text{for}&~t\in(0,T),\\
    \bm{A}_h(q_h)&\coloneqq \left(\langle A(q_h)\varphi_j,\varphi_i \rangle_{V',V}\right)_{i,j \in \{1, \dots, N_V\}}&\text{for}&~q_h\in Q_h. 
\end{align*}
For discretization in time, we use the methods of lines \cite{ern2004theory} on an equidistant temporal grid $\{0 \eqqcolon t^0 < \dots < t^K \coloneqq T\}$ of $[0,T]$, with $K\in \N$ and $\Delta t\coloneqq t^{k}-t^{k-1}$ for all $k \in \{1, \dots K\}$. We denote $\mathbb{K} \coloneqq \{1, \dots, K\}$, and write $v = [v^1, \dots, v^K] \in V_h^K$ to represent a discrete trajectory with values in $V_h$. Here, $V_h^K$ denotes the product space of $V_h$. In the following, we identify $v$ with its DoF-matrix $\bm{v}=[\bm{v}^1,\dots,\bm{v}^K]\in \R^{N_V\times K}$ where $v^k = \sum_{i = 1}^{N_V} \bm{v}^k_{i}\varphi_i$. We equip these spaces with the discrete inner products
\begin{equation}
    \label{eq:disc_bochner_prods}
    \langle v, w \rangle_{V_h}\coloneqq \bm{v}^\top \bm{M}_V \bm{w} \text{ and } \langle v, w \rangle_{V_h^K}\coloneqq \Delta t \sum_{k = 1}^K \langle v^k, w^k \rangle_{V_h}
\end{equation}
respectively. Analogous definitions apply to the spaces $Q_h$ and $C_h$.
For the discretization of the state equation \eqref{eq:weak_problem}, let $q_h \in \qadK \subset Q_h^K$, where the discrete admissible set is defined as
\begin{equation}
    \label{eq:discCons}
    \qadK \coloneqq \bigg\{q\in  Q_h^K  \ \bigg| \sum\limits_{k=1}^{K} \mathbbm{1}_{(t_{k-1},t_k)}q^k\in \Qad\bigg\} \subset Q^K_h.
\end{equation}
The state DoF-matrix $\bm{u}_h \in \mathbb{R}^{N_V \times K}$ for the discrete state $u_h \in V_h^K$ is obtained as the unique solution to the discrete state equation:
\begin{equation}
    \label{eq:FOM_primal_PDE}
    \frac{1}{\Delta t} \bm{M}_H (\bm{u}_h^k - \bm{u}_h^{k-1}) + \bm{A}_h(q_h^k) \bm{u}_h^k = \bm{L}_h^k\quad \text{for } k \in \mathbb{K},
\end{equation}
with initial condition $\bm{u}_h^0 = 0$ and $\bm{L}_h^k\coloneqq \bm{L}_h(t^k)$. We denote the discrete solution operator by $\Sol_h : \qadK \rightarrow V_h^K$, returning for $q_h \in Q_h^K$ the solution $u_h=\Sol_h(q_h)$ to \eqref{eq:FOM_primal_PDE} and define the discrete observation operator $\C_h: V_h \rightarrow C_h$ for $v_h \in V_h$ as $w_h = \C_h v_h$ via
\begin{align*}
    & \bm{w}_h \coloneqq \bm M_C^{-1}
	\left(\langle \C v_h,c_j\rangle_{C}\right)_{j\in \{1,\dots,N_C\}}.
\end{align*} 
The discrete forward operators is given as $\F_h := \C_h \circ \Sol_h : \qadK \rightarrow C_h^K$, in the sense that $\F_h(q_h^k) = \C_hu^k_h$ for all $k \in \mathbb{K}$.
Define analogue the projection $y_h^\delta\in C_h^K$ of $y^\delta$:
\begin{align*}
    & \bm{y}^{\delta,k}_h\coloneqq \bm M_C^{-1}
	\left(\langle y^\delta(t^k),c_j\rangle_{C}\right)_{j\in \{1,\dots,N_C\}}
	\text{ for }k\in \mathbb K.
\end{align*}
To define the discrete objective in \eqref{eq:minJ}, we discretize the bilinear, linear and constant parts of the cost function as follows:
\begin{align*}
     &\bm{C}_h \coloneqq \bm \left(\langle \C_h\varphi_j, \C_h\varphi_i \rangle_{C_h}\right)_{i,j \in \{1, \dots, N_V\}},\ {\bm{Cy}^{\delta,k}_h}\coloneqq \bm \left(\langle \C_h\varphi_j, y_{h}^{\delta,k}\rangle_{C_h}\right)_{1\le j \le N_V} \text{ for }k\in \mathbb K,\\
     &c_1(\bm{y}_h^{\delta}) \coloneqq {\Delta t} \sum_{k = 1}^K\tfrac{1}{2}\big({\bm{y}_h^{\delta,k}}^\top \bm M_C {\bm{y}_h^{\delta,k}} \big).
\end{align*}
Then, with $\Sol_h(q_h)=u_h$ for $q_h\in Q_h^K$, the discrete objective is given by
\begin{align}
    \label{eq:J_h}
    \begin{aligned}
        J_h(q_h) &\coloneqq \frac{\Delta t}{2} \sum_{k = 1}^K \|\mathcal{F}_h(q_h)^k - y_{h}^{\delta,k} \|^2_{C_h},\\
        &\hspace{0.5mm}=c_1(\bm{y}_h^{\delta})+{\Delta t}  \sum_{k = 1}^K  \tfrac{1}{2}\big({\bm{u}_h^{k\top}}\bm C_h \bm{u}_h^k \big)-\big({\bm{u}_h^{k\top}} \bm{Cy}^{\delta,k}_h\big).
    \end{aligned}
\end{align}
The discrete counterpart of the adjoint equation \eqref{eq:adjoint_eq} reads for $p_h\in V_h^K$
\begin{align*}
    \frac{1}{\Delta t}\bm M_H(\bm{p}_h^k - \bm{p}_h^{k+1}) + \bm A_h(q^k_h)^\top\bm{p}^k_h & = -\bm C_h \bm{u}_h^k +\bm{Cy}^{\delta,k}_h \quad \forall k \in \mathbb{K},
\end{align*}
where $\bm{p}_h^{K+1} = 0$. The discrete gradient $\nabla J_h(q_h) \in Q_h^K$ is given component-wise by
\begin{equation*}
    \bm{\nabla J}_h(q_h)^k = \bm M_Q^{-1}\bm B_h(u_h^k)^\top \bm{p}^k_h \quad \forall k \in \mathbb{K}
\end{equation*}
where 
\begin{equation*}
\bm{B}_h(u_h^k)\coloneqq \left( \langle A(q_j) u_h^k, \varphi_i \rangle_{V',V} \right)_{i \in \{1, \dots, N_V\}, j \in \{1, \dots, N_Q\}} \in \mathbb{R}^{N_V \times N_Q}. 
\end{equation*}
The resulting discretized optimization problem corresponding to \eqref{eq:minJ} reads
\begin{equation}
	\label{eq:minJ_FOM}
	\tag{$\mathbf{IP}_{h}$}
	\min J_h(q_h) \quad \text{s.t.}\quad q_h \in \qadK.
\end{equation}
The procedure outlined in Section~\ref{ssec:realization_IRGNM} for solving the minimization problem \eqref{eq:minJ}, can be directly transferred to a discrete formulation. Setting the constant parts as ${{\tilde{\bm{q}}}_{\circ,h}^{k} \coloneqq {\bm{q}_{\circ,h}^{k}} - {\bm{q}_{h}^{k}}}$, $\bm{C}\tilde{\bm{y}}^{\delta,k}_h={\bm{Cy}^{\delta,k}_h}-\bm{C}_h \bm{u}_h^{k}$, and
\begin{align*}
    &c_2(\tilde{\bm{y}}^{\delta,k}_h) \coloneqq c_1(\bm{y}_h^{\delta})+ {\Delta t} \sum_{k = 1}^K\tfrac{1}{2}\big({\bm{u}_h^{k\top}} \bm C_h {\bm{u}_h^{k}} \big)-\big({\bm{u}_h^{k\top}} \bm{Cy}^{\delta,k}_h\big),\\
    &c_3({\tilde{\bm{q}}}_{\circ,h})\coloneqq  {\Delta t} \sum_{k = 1}^K\tfrac{1}{2} \tilde{\bm{q}}_{\circ,h}^{k\top} \bm{M}_Q \tilde{\bm{q}}_{\circ,h}^k,
\end{align*}
the discrete, linearized reduced objective $\widetilde{J}_h$ is defined as
\begin{align*}
    \widetilde{J}_h(d_h; q_h, \alpha)  \coloneqq\, &  c_2(\tilde{\bm{y}}^{\delta,k}_h)+\Delta t \sum_{k=1}^K \left[
    \tfrac{1}{2} \tilde{\bm{u}}_h^{k\top} \bm{C}_h \tilde{\bm{u}}_h^k 
    - \tilde{\bm{u}}_h^{k\top} \bm{C}\tilde{\bm{y}}^{\delta,k}_h
    \right] \\
    & + \alpha c_3(\tilde{\bm{q}}_{\circ,h}) + \alpha \Delta t \sum_{k=1}^K \left[
    \tfrac{1}{2} \bm{d}_h^{k\top} \bm{M}_Q \bm{d}_h^k 
    - \bm{d}_h^{k\top} \bm{M}_Q \tilde{\bm{q}}_{\circ,h}^k  \right].
\end{align*}
The discrete, linearized subproblem \eqref{eq:IRGNMscheme_minimize} at $q^\iteridx_h$ is given as:
\begin{align}
    \label{prob:lin_prob_FOM}
    \tag{$\widetilde{\mathbf{IP}}_{\alpha,h}$}
    d^\iteridx_h\coloneqq\argmin \widetilde{J}_h(d_h; q^\iteridx_h, \alpha^{\iteridx})\quad \text{s.t.}\quad q^\iteridx_h + d_h \in  \qadK,
\end{align}
for $\alpha^{(i)} > 0$. The update direction $d^\iteridx_h$ is accepted and the iterate updated, i.e. ${q_h^{(i+1)} := q_h^\iteridx + d^\iteridx_h}$, if, analogue to \eqref{eq: choice of alpha}, hold
\begin{equation}
	\label{eq:choice_of_alpha_FOM}
	\theta J_h(q_h^\iteridx)\leq 2  \widetilde{J}_h(d^\iteridx_h; q^\iteridx_h, 0)\leq \Theta J_h(q_h^\iteridx),
\end{equation}
for $0<\theta<\Theta<2$. Otherwise, the regularization parameter $\alpha^\iteridx > 0$ is adjusted as described in Section~\ref{sec:introduction}. To provide a numerical procedure for solving \eqref{prob:lin_prob_FOM}, we have to compute the gradient of the linearized objective. Therefore, define the discrete linearized state $\tilde{u}_h \in V_h^K$ as the unique solution of
\begin{align}
	\label{eq:state_prob_FOM}
    \frac{1}{\Delta t}\bm{M}_H(\tilde{\bm{u}}_h^{k} - \tilde{\bm{u}}_h^{k-1}) + \bm{A}_h(q_h^k) \tilde{\bm{u}}_h^k + \bm{B}_h(u_h^k) \bm{d}_h^k = 0 \quad \forall k \in \mathbb{K},
\end{align}
with initial condition $\tilde{\bm{u}}_h^0 = 0$. The discrete linearized adjoint state $\tilde{p}_h \in V_h^K$ satisfies
\begin{align}
	\label{eq:adj_prob_FOM}
    \frac{1}{\Delta t}\bm{M}_H(\tilde{\bm{p}}_h^k - \tilde{\bm{p}}_h^{k+1}) + \bm{A}_h(q_h^k)^\top \tilde{\bm{p}}_h^k 
    = -\bm{C}_h \tilde{\bm{u}}_h^k + \bm{C}\tilde{\bm{y}}_h^{\delta,k} \quad \forall k \in \mathbb{K},
\end{align}
with terminal condition $\tilde{\bm{p}}_h^{K+1} = 0$. The discrete linearized gradient is then given by
\begin{equation*}
    \bm{\nabla_d} \widetilde{\bm{J}}_{h}(d_h; q_h, \alpha)^k = \bm{M}_Q^{-1} \bm{B}_h(u_h^k)^\top \tilde{\bm{p}}_h^k + \alpha (\bm{d}_h^k - \tilde{\bm{q}}_{\circ,h}^k), \quad \forall k \in \mathbb{K}.
\end{equation*}
Using this gradient, the linear problem can be solved numerically, via projected gradient decent, cf.~\cite{azmi2023nonmonotone}. Thus, problem~\eqref{eq:minJ_FOM} can be solved by applying Algorithm~\ref{algo:IRGNM}, where \eqref{prob:lin_prob_FOM} is iteratively solved. In the following, this procedure will be referred to as \emph{FOM-IRGNM}.

\begin{remark}
    \label{rem:linear_problem_time_indep}
    As noted in Remark~\ref{rem:example_time_indep}, the formulation above naturally includes the case of parameter fields that are constant in time by setting $q_h^k \coloneqq\tilde{q}_h \in Q_h$ for all $k \in \mathbb{K}$. In this case, the linearized subproblem is posed on $Q_h$, and the regularization acts on the deviation from a reference parameter ${q}_{\circ,h} \in Q_h$. The set of admissible parameter $\qad \subset Q_h$ is defined analogue to \eqref{eq:discCons}. The update $d_h^\iteridx$ at $q_h^\iteridx \in Q_h$ is computed as
    \begin{equation*}
        d^\iteridx_h \coloneqq \argmin \widetilde{J}_h(d_h; q^\iteridx_h, \alpha^{\iteridx}) \quad \text{s.t.}\quad q^\iteridx_h + d_h \in \qad,
    \end{equation*}
    for 
    \begin{align*}
        \widetilde{J}_h(d_h; q_h, \alpha)&\coloneqq c_2(\tilde{\bm{y}}^{\delta,k}_h)+\Delta t \sum_{k=1}^K \left[
    \tfrac{1}{2} \tilde{\bm{u}}_h^{k\top} \bm{C}_h \tilde{\bm{u}}_h^k 
    - \tilde{\bm{u}}_h^{k\top} \bm{C}\tilde{\bm{y}}^{\delta,k}_h
    \right]  \\
        &\,\quad + \alpha \left[
        \tfrac{1}{2} \bm{d}_h^\top \bm{M}_Q \bm{d}_h - \bm{d}_h^\top \bm{M}_Q \tilde{\bm{q}}_{\circ,h} + \tfrac{1}{2} \tilde{\bm{q}}_{\circ,h}^{\top} \bm{M}_Q \tilde{\bm{q}}_{\circ,h} \right],
    \end{align*}
    where $\tilde{\bm{q}}_{\circ,h} = \bm{q}_{\circ,h} - \bm{q}_h$. The corresponding discrete gradient $\nabla \widetilde{J}_h(d_h; q_h, \alpha) \in Q_h$ reads
    \begin{equation*}
        \bm{\nabla_d} \widetilde{\bm{J}}_{h}(d_h; q_h, \alpha) = \Delta t \sum_{k=1}^K \bm{M}_Q^{-1} \bm{B}_h(u_h^k)^\top \tilde{\bm{p}}_h^k + \alpha (\bm{d}_h - \tilde{\bm{q}}_{\circ,h}).
    \end{equation*}
\end{remark}

\subsection{Reduced basis model}
\label{sec:modelreduction}
%
\noindent
In this section, we introduce the RB-ROM with reduction in both the parameter and state space. In each iteration of the FOM-IRGNM, the linearized local optimization problem \eqref{prob:lin_prob_FOM} must be solved at least once. Hence, applying MOR is a natural choice to speed up the solution process. However, replacing the high-fidelity setup with reduced-order models necessitates an error control mechanism, such as an error-aware trust-region framework \cite{Keil2021nonco-54293, YueMeerbergen2013}, to ensure the accuracy of the reduced solutions.

As noted in~\cite{kartmann_adaptive_2023}, the high dimensionality of the parameter space poses a significant challenge for constructing certified reduced-order models via projection. The core difficulty lies in the structure of the bilinear form $\bm A_h(q_h)$ for $q_h \in Q^K_h$. Due to the linear dependence of $\bm A_h(q_h)$ on $q_h$, it admits an affine decomposition of the form
\begin{equation}
    \label{eq:affine_decomp}
    \bm A_h(q^k_h)= \bm A_{h}(0) + \sum^{N_Q}_{j = 1} \bm{q}^k_{h, j} \bm A_{h, j}
\end{equation}
for $k \in \mathbb{K}$, where $\bm A_{h, j} \coloneqq \bm A_h(q_j) \in \mathbb{R}^{N_V \times N_V}$ for $j = 1, \ldots, N_Q$. The dimension of the parameter space $Q_h$ can be arbitrarily large, which complicates the use of reduced basis methods. Such methods rely usually \textit{on low-dimensional} affine decompositions to enable efficient offline-online splitting. As highlighted in~\cite{kartmann_adaptive_2023}, a large number of affine components render both assembly and projection of $\bm A_h(q^k_h)$ computationally prohibitive, limiting the applicability of efficient offline-online decomposition, unless parameter reduction is applied.

To address these limitations, we extend the methodology of~\cite{kartmann_adaptive_2023}, originally developed for parameter identification problems governed by elliptic PDEs, to the setting of parabolic problems. In this approach, both the state and parameter spaces are replaced with adaptively constructed reduced basis spaces, and the resulting reduced basis models are employed to solve the linearized optimization problem~\eqref{prob:lin_prob_FOM}. By simultaneously reducing the state and parameter spaces, the computationally expensive assembly of high-dimensional system matrices is avoided, as the reduced-order model operates entirely within low-dimensional subspaces. 

For parabolic PDE constraints, we improve the reduced basis enrichment strategy of~\cite{kartmann_adaptive_2023} by incorporating proper orthogonal decomposition (POD), which allows for a trade-off between local approximation accuracy and computational efficiency. However, this enhancement requires additional safeguards and fallback mechanisms within the algorithm to ensure robust optimization performance. Moreover, we explicitly enforce the set of admissible parameters $\qadK$ using proximal optimization techniques, in contrast to~\cite{kartmann_adaptive_2023}, where this constraint was only treated implicitly.

In the following, we describe the implementation of the algorithm from~\cite{kartmann_adaptive_2023}, adapted to parabolic PDE-constrained problems. We begin by reformulating the high-dimensional optimization problem in a reduced setting and introducing a posteriori error estimators for the reduced objective functional. Next, we present a modified version of the algorithm from~\cite{kartmann_adaptive_2023}, referred to as TR-IRGNM, which integrates the adaptively constructed reduced models into an error-aware trust-region framework with additional safeguards. Finally, we introduce the enrichment strategy used for iteratively constructing the reduced basis spaces.

Let a reduced state space $V_r=\text{span}(\tilde{\varphi}_1,\dots,\tilde{\varphi}_{n_V}) \subset V_h$ and a reduced parameter space $Q_r=\text{span}(\tilde{q}_1,\dots,\tilde{q}_{n_Q}) \subset Q_h$ be given, with dimensions $n_Q  \in \mathbb{N}$ and $n_V \in \mathbb{N}$. We denote the basis matrices as $\bm \Psi_Q \in \R^{N_Q\times n_Q}$, $\bm \Psi_V \in \R^{N_V\times n_V}$ and assume that 
\begin{equation*}
\bm \Psi_V^\top \bm M_V \bm \Psi_V = \bm I_{n_V}
\text{ and } 
\bm \Psi_Q^\top \bm M_Q \bm \Psi_Q= \bm I_{n_Q}.     
\end{equation*}
The corresponding reduced objects are given as
\begin{align*}
    &\bm{M}_{Q,r} \coloneqq \bm I_{n_Q}, \quad \bm{M}_{V,r} \coloneqq \bm I_{n_V},\quad \bm{M}_{H,r} \coloneqq \bm \Psi_V^\top \bm{M}_{H} \bm \Psi_V,\quad  \bm{L}_r^k= \bm\Psi_V^\top \bm{L}_h^k\quad \forall k\in \mathbb K.
\end{align*}
For $q_r\in Q_r$, the reduced operator $\bm{A}_r(q_r)$ can be expressed as
\begin{equation*}
    \bm{A}_r(q_r) = \bm{A}_r(0) + \sum_{j=1}^{n_Q} \bm{q}_{r,j} \bm{A}_r(\tilde{q}_j), \quad \text{for } q_r \in Q_r.
\end{equation*}
Due to the low dimensionality of the reduced parameter space, the number of affine components is small compared to that of $\bm{A}_h(q_r)$, which significantly reduces the computational complexity of assembling $\bm{A}_r(q_r)$, particularly after precomputing the matrices $\bm{A}_r(\tilde{q}_j)$ for $j = 1, \dots, n_Q$. Given a reduced parameter ${q_r \in Q^K_r}$, that we identify with its DoF-matrix ${\bm q_r \in \R^{n_Q\times K}}$, the RB approximation of the state is given as the solution of the reduced primal problem
\begin{equation*}
    \frac{1}{\Delta t}\bm M_{H,r}(\bm{u}_r^k - \bm{u}_r^{k-1}) + \bm A_r(q^k_r)\bm{u}^k_r = \bm L_r^k \quad \forall k \in \mathbb{K},
\end{equation*}
with $\bm u_r^0 = 0$. We write $u_r=\Sol_r(q_r)$ and define $\F_r=\C_h\circ\Sol_r : \qadrK \rightarrow C_h^K$. The reduced admissible set $\qadrK \subset Q^K_r$ is defined similar to \eqref{eq:discCons}. Remark that $\Sol_r$ is well-defined, due to Assumption \ref{ass:bilinear_form}. The reduced objective for $q_r\in Q_r^K$ is thus given by
\begin{align}
    \label{eq:J_r}
    J_r(q_r) \coloneqq 
    & c_1(\bm{y}_h^{\delta}) + {\Delta t}\sum_{k = 1}^K \tfrac{1}{2}\big({\bm{u}_r^{k\top}} \bm{C}_r \bm{u}_r^k \big)-\big({\bm{u}_r^{k\top}} \bm{C}\bm{y^{\delta,k}}_r \big)
\end{align}
with
\begin{align*}
    &\bm{C}_r=\bm\Psi_V^\top \bm C_h \bm\Psi_V,\quad \bm{Cy}^{\delta,k}_r=\bm\Psi_V^\top \bm C\bm{y}_h^{\delta,k} \text{ for }k\in \mathbb K.
\end{align*}
Remark, that online evaluating $J_r(q_r)$, does not require any full-order computations.
The adjoint equation for $p_r\in V_r^K$ reads:
\begin{equation*}
    \frac{1}{\Delta t}\bm M_{H,r}(\bm{p}_r^k - \bm{p}_r^{k+1}) + \bm A_r(q^k_r)^\top\bm{p}^k_r = -\bm C_r\bm{u}_r^k +\bm{Cy}^{\delta,k}_r \quad \text{for all }k \in \mathbb{K} 
\end{equation*}
with $\bm{p}_r^{K+1} = 0$. Hence, the discrete gradient $\nabla J_r(q_r)\in Q_r^K$ can be computed by
\begin{equation*}
    \bm{\nabla J}_r(q_r)^k = \bm B_r(u_r^k)^\top \bm{p}^k_r  \quad \forall k \in \mathbb{K}
\end{equation*}
where $\bm B_r(u_r^k)=\bm\Psi_V^\top \bm B(u_r^k)\bm\Psi_Q\in \R^{n_V\times n_Q}$. The reduced optimization problem can thus be stated as
\begin{equation}
\label{prob:prob_ROM}
	\tag{$\textbf{IP}_{r}$} 
	\min{J_r(q_r)}\text{ s.t. } q_r \in \qadrK.
\end{equation}
Setting ${\tilde{\bm{ q}}}_{r,0}^{k}=\bm \Psi_Q^\top \bm M_Q\tilde{\bm{ q}}_{h,0}^{k}$, and $\bm{C}\tilde{\bm{y}}^{\delta,k}_r={\bm{Cy}^{\delta,k}_r}-\bm{C}_r \bm{u}_r^{k}$ and the constant parts to
\begin{align*}
    &c_2(\tilde{\bm{y}}^{\delta,k}_r) \coloneqq c_1(\bm{y}_h^{\delta})+ {\Delta t} \sum_{k = 1}^K\tfrac{1}{2}\big({\bm{u}_r^{k\top}} \bm C_r {\bm{u}_r^{k}} \big)-\big({\bm{u}_r^{k\top}} \bm{Cy}^{\delta,k}_r\big),\\
    &c_3({\tilde{\bm{q}}}_{\circ,r})\coloneqq  {\Delta t} \sum_{k = 1}^K\tfrac{1}{2} \tilde{\bm{q}}_{\circ,r}^{k\top} \tilde{\bm{q}}_{\circ,r}^k,
\end{align*}
the linearized reduced objective $\widetilde{J}_r$ is defined as
\begin{align*}
    \widetilde{J}_r(d_r; q_r, \alpha)  \coloneqq &  c_2(\tilde{\bm{y}}^{\delta,k}_r)+\Delta t \sum_{k=1}^K \left[
    \tfrac{1}{2} \tilde{\bm{u}}_r^{k\top} \bm{C}_r\tilde{\bm{u}}_r^k 
    - \tilde{\bm{u}}_r^{k\top} \bm{C}\tilde{\bm{y}}^{\delta,k}_r
    \right] \\
    & +  \alpha c_3({\tilde{\bm{q}}}_{\circ,r})+ \alpha \Delta t \sum_{k=1}^K \left[
    \tfrac{1}{2} \bm{d}_r^{k\top} \bm{d}_r^k 
    - \bm{d}_r^{k\top} \tilde{\bm{q}}_{\circ,r}^k  \right],
\end{align*}
where we have used $\bm M_{Q,r}=\bm I_{n_Q}$. The reduced linearized subproblem is consequently given by
\begin{align}
\label{prob:lin_prob_rom}
	\tag{$\widetilde{\textbf{IP}}_{\alpha, r}$}
    d^\iteridx_r	:= \argmin \tilde{J}_r(d_r; q^\iteridx_r, \alpha^\iteridx) \text{ s.t. } q^\iteridx_r + d_r \in \qadrK.
\end{align}
and the reduced discrepancy principle for $0 < \theta < \Theta < 2$ by
\begin{equation}
	\label{eq:choice_of_alpha_ROM}
	\theta J_r(q_r^\iteridx)\leq 2\widetilde{J}_r(d^\iteridx_r; q^\iteridx_r, 0) \leq \Theta J_r(q_r^\iteridx).
\end{equation}
The coefficients of the gradient of $\tilde J_r$ at $q_r \in \qadrK$ are given by 
\begin{equation*}
    \bm{\nabla_d} \widetilde{\bm{J}}_r(d_r; q_r, \alpha)^k = \bm B_r(u_r^k)^\top \tilde{\bm{p}}^k_r + \alpha(\bm{d}_r^k-\tilde {\bm{q}}_{\circ, r}^k) \quad \forall k \in \mathbb{K}.
\end{equation*} 
The reduced linearized state $\tilde u_r\in V_r^K$ is obtained by solving
\begin{align}
	\label{eq:state_prob_ROM}
    \frac{1}{\Delta t}\bm M_{H,r}(\tilde{\bm{ u}}_r^{k} - \tilde{\bm{ u}}_r^{k-1}) + \bm A_r(q^{k}_r)\tilde{\bm{ u}}_r^{k} + \bm B_r(u_r^{k})\bm{d}^{k}_r = 0 \quad \forall k\in \mathbb{K},  
\end{align}
with $\tilde{\bm{u}}^0_r = 0$ and $u_r = \Sol_r(q_r)$ and reduced linearized adjoint state $\tilde p_r\in V_r^K$ by solving
\begin{align}
	\label{eq:adj_prob_ROM}
    \frac{1}{\Delta t}\bm M_{H,r}(\tilde{\bm{ p}}_r^{k} -\tilde{\bm{ p}}_r^{k+1}) + \bm A_r(q^{k}_r)^\top\tilde{\bm{ p}}_r^{k} = -\bm C_r \tilde{\bm{u}}^{k}_h + \bm{C}\tilde{\bm{y}}^{\delta,k}_r  \quad \forall k\in \mathbb{K},
\end{align}
with $\tilde{\bm{p}}_r^{K+1}=0$. 


%
\subsubsection{Error estimation}
\label{ssec:errorest}
To quantify the error between $J_r(q_r)$ and $J_h(q_r)$ for $q_r \in Q_r^K$, we employ a residual-based a posteriori error estimator. To this end, we define the primal residual ${\bm{r}^k_{\mathrm{pr}}(u; q) \in V_h'}$ for $k \in \mathbb{K}$, $q \in Q_h^K$ and $u \in V_h^K$:
\begin{equation}
	\label{eq:primal_residual}
    \bm{r}^k_\text{pr}(u; q) := \bm{L}_h^k - \bm{A}_h(q^k)\bm{u}^k - \frac{1}{\Delta t}\bm{M}_h(\bm{u}^k - \bm{u}^{k-1}) \in \mathbb{R}^{N_V}
\end{equation}
Similarly for $p \in V_h^K$, define the adjoint residual $\bm{r}^k_\text{ad}(u, p; q)\in V'_h$, by setting
\begin{equation}
	\label{eq:adjoint_residual}
    \bm{r}^k_\text{ad}(u,p; q) := -\bm{C}_h \bm{u}^k + \bm{C}\bm{y}_h^{\delta,k} -\bm{A}_h(q^k)^\top\bm{p}^k - \frac{1}{\Delta t}\bm{M}_h(\bm{p}^k - \bm{p}^{k+1}) \in \mathbb{R}^{N_V},
\end{equation}
for all $k \in \mathbb{K}$. Based on these definitions, we are now in a position to derive an a posteriori error estimator for $J_r$.
\begin{proposition}
\label{prop:a_posterori_error_est}
Let $q_r \in \qadrK$ and $a_{q^k_r} \geq  a >0$ be the $q$-dependent coercivity constant for $\bm{A}_h(q_r^k)$ from Assumption \ref{ass:bilinear_form}.
Then it holds 
\begin{equation*}
|J_r(q_r) - J_h(q_r)| \leq \Delta^{J}(q_r),
\end{equation*}
where the right-hand side is given by
\begin{align*}
    \Delta^{J}(q_r) &:= \left( \Delta t \sum_{k = 1}^K \|\bm{r}^k_\text{ad}(u_r,p_r; q_r)\|^2_{V'_h}\right)^{\frac{1}{2}} \frac{\Delta^{\text{pr}}(q_r)}{\sqrt{a_{q_r}}} + \frac{\|\mathcal C\|^2_{\mathcal L (V, C)}}{2 a_{q_r}} \Delta^{\text{pr}}(q_r)^2\text{ and }\\    
    \Delta^{\text{pr}}(q_r) & := \left(\Delta t \sum_{k = 1}^K \frac{1}{a_{q_r}}\|\bm{r}^k_\text{pr}(u_r; q_r)\|_{V'_h}^2\right)^{\frac{1}{2}},
\end{align*}
where $a_{q_r} := \min_{k \in \mathbb K} a_{q^k_r}$.
\end{proposition}
\begin{proof}
This bound follows directly from the argumentation in \cite[Lemma 8, Theorem 9]{QGVW17}, where the last term in \cite[Theorem 9]{QGVW17} can be omitted, because the reduced primal $u_r$ and adjoint states $p_r$ are elements of the same space $V_r^K$.
\end{proof}
\begin{remark}
\label{rem:zero_error_cond}
We have the following interpolation property: if $q_r \in Q^K_r$ and $u_h(q_r), p_h(q_r) \in V_r^K$ hold: $\Delta^J(q_r) = 0$.    
\end{remark}
\section{Error-aware trust region IRGNM}
\label{sec:TR-IRGNM}
The key idea of the trust-region method is to replace the original optimization problem~\eqref{eq:minJ} by a sequence of reduced subproblems of the form
\begin{equation}
\label{eq:TR_Problem}
q^{(i+1)}_r \coloneqq \argmin J_r^{(i)}(q_r) \quad \text{s.t.}\quad q_r \in T^{(i)},
\end{equation}
starting from an initial guess $q_r^{(0)} \in \qadK$ and continuing until a termination criterion is met. We define the iteration index set as $\mathbb{I} := \{0, \dotsc, i_*(\delta, \yd)\}$. Here, $J_r^{(i)}$ denotes the reduced objective functional~\eqref{eq:J_r}, defined using the RB-ROMs based on the reduced spaces $Q_r^{(i)} \subset Q_h$ and $V_r^{(i)} \subset V_h$. Each subproblem is (approximately) solved using the IRGNM, as outlined in Algorithm~\ref{algo:IRGNM}, generating a sequence $\{q_r^{(i,l)}\}_{l = 0}^{L^{(i)}} \subset \Qridx$. Figure \ref{fig:TR_illustration} shows a schematic representation of the abstract trust region optimization. The reduced spaces are updated, by snapshots at $q^{(i+1)}_r$, re-adapting the RB-ROM locally to the solution manifold; see Section~\ref{ssec:const_rb} for details.\\
Since snapshot-based reduced models are typically only locally accurate, the optimization is  restricted to error-aware trust regions, cf. ~\cite{QGVW17,Keil2021nonco-54293,Banholzer2022Trust,klein2025multifidelitylearningreducedorder}. Following these approaches, we define the trust region as
\begin{equation*}
T^{(i)} \coloneqq \left\lbrace q_r \in \qadriK \,\bigg|\, \errest(q_r) := \frac{\Delta^{J,(i)}(q_r)}{| J_r^{(i)}(q_r)|} \leq \eta^{(i)} \right\rbrace,
\end{equation*}
where the tolerance $\eta^{(i)} > 0$ controls the radius of the trust region, and $\qadriK \subset \Qad$ denotes the set of admissible reduced parameters associated with the reduced space $Q_r^{(i)}$, analogous to~\eqref{eq:discCons}.

To ensure convergence with respect to the full-order objective $J_h$, additional checks on the minimizers of~\eqref{eq:TR_Problem} are necessary, since local minimization of the surrogate $J_r^{(i)}$ may not guarantee sufficient global reduction in $J_h$. For a given surrogate and trust region, we therefore compute trial solutions $\qtrial$ to the subproblem and accept them only if they satisfy additional decrease conditions (see below). Once a sufficiently accurate local solution $q_r^{(i+1)}$ is obtained, the reduced spaces are updated by enriching the basis with snapshots computed at this solution. The algorithm proposed in this work consists of three main steps, which are outlined and discussed in detail below.

\begin{figure}[t!]
\centering
\includegraphics[width=0.80\textwidth]{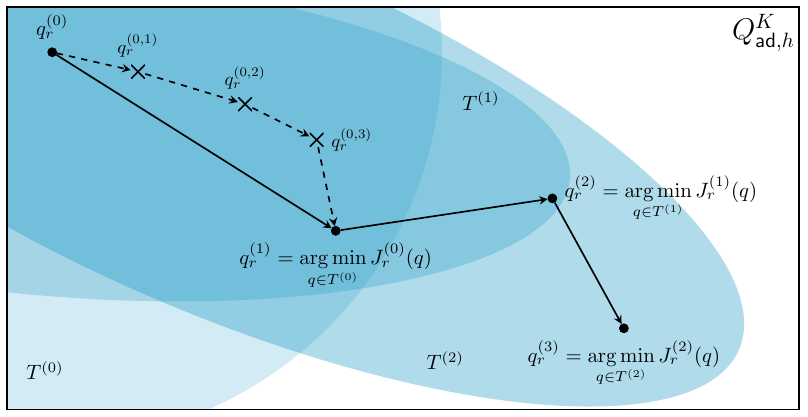}
\caption{\label{fig:TR_illustration} Schematic representation of the abstract trust region algorithm.}
\end{figure}
\subsection{Error-aware reduced basis trust region algorithm}
\label{ssec:rb_TR_IRGNM}
\textbf{Computation of the AGC:} As an initial guess $q_r^{(i,0)} \in Q^\iteridx_r$ for the $i$-th subproblem, we select the approximate solution from the previous subproblem, projected onto the updated parameter space $Q_r^{(i)}$. Since, we will construct the spaces such that always hold ${q_r^{(i)} \in Q_r^{(i)}}$, we can set
$q_r^{(i,0)} = q_r^{(i)}$, see Section~\ref{ssec:const_rb}. Additonally, to ensure well-definedness, it is has to be verfied that the initial guess lies in the current trust region, i.e., $q_r^{(i,0)} \in T^\iteridx$. Note that Remark~\ref{rem:zero_error_cond} guarantees that always reduced spaces $Q_r^\iteridx$ and $V_r^\iteridx$ can be constructed to satisfy this condition. If $q_r^{(i,0)} \notin T^\iteridx$, the reduced spaces will be adapted by enriching the bases with the most relevant POD modes of $q_r^{(i)}$, $u_h(q_r^{(i)})$ and $p_h(q_r^{(i)})$, respectively; see Section~\ref{ssec:const_rb} for details. This procedure is repeated iteratively, if necessary, until a suitable $q_r^{(i,0)}$ is obtained.

The first iterate $q_r^{(i,1)}$ in each subproblem is computed as the approximate generalized Cauchy point (AGC). The AGC, denoted by ${q^\iteridx_{\text{AGC},r} \in Q_r^\iteridx}$, is defined as
\begin{equation*}
q^\iteridx_{\text{AGC},r} = P_{\qadriK}\left(q_r^{(i, 0)} - t \nabla J^\iteridx_r(q_r^{(i, 0)})\right),
\end{equation*}
where $t > 0$ is chosen such that $q^\iteridx_{\text{AGC},r} \in T^\iteridx$ and the following sufficient decay condition holds:
\begin{equation}
\label{eq:armijo_decay_condition}
J^\iteridx_r(q^\iteridx_{\text{AGC},r}) - J^\iteridx_r(q_r^{(i, 0)}) \leq - \frac{\kappa_\text{arm}}{t} \|q^\iteridx_{\text{AGC},r} - q_r^{(i, 0)}\|^2_{Q^K_h},
\end{equation}
with some $\kappa_\text{arm} > 0$. The step size $t$ is determined by a backtracking line search initialized with
\begin{equation*}
\bar{t} = \min \{0.5 \|\nabla J_h(q^{(0)})\|^{-1}_{Q^K_h}, 1\}.
\end{equation*}
To assert a decay in $J_h$, we additionally require that
\begin{equation}
\label{eq:AGC_decay_cond}
J^\iteridx_r(q^\iteridx_{\text{AGC},r}) < J_h(q_r^{(i)}).
\end{equation}
This condition is automatically satisfied if 
\begin{equation*}
J^\iteridx_r(q_r^{(i, 0)}) = J_h(q_r^{(i, 0)}) \quad \text{and} \quad \nabla J^\iteridx_r(q_r^{(i, 0)}) = \nabla J_h(q_r^{(i, 0)}).    
\end{equation*}
Therefore, if condition~\eqref{eq:AGC_decay_cond} is violated, the reduced spaces are iteratively enriched by including the most relevant POD modes of $\nabla J_h(q_r^{(i, 0)})$, $u_h(q_r^{(i)})$ and $p_h(q_r^{(i)})$ into the respective bases and recomputing the AGC.
This procedure is repeated until condition~\eqref{eq:AGC_decay_cond} is satisfied, which can be ensured after finitely many updates, see Remark~\ref{rem:zero_error_cond}. Is the resulting $q^\iteridx_{\text{AGC},r}$ already close to the boundary of the trust region, i.e.,
\begin{equation}
\label{eq:TR_condition}
\beta_1 \eta^\iteridx \leq \errest(q^\iteridx_{\text{AGC},r}),
\end{equation}
for some $\beta_1 \in (0,1)$, then the AGC it is directly accepted as iterate, i.e. ${q^{(i+1)}_r := q^\iteridx_{\text{AGC},r}}$. Condition \eqref{eq:TR_condition} is a common termination criteria for a trust-region algorithm to avoid excessive computational effort near the boundary, where the accuracy of the RB-ROM degrades, cf.~\cite{QGVW17,Keil2021nonco-54293}.

\vspace{\baselineskip}
\textbf{The trust region subproblem:} If condition~\eqref{eq:TR_condition} does not hold, we apply the IRGNM using the RB-ROM to solve the TR subproblem \eqref{eq:TR_Problem}. Starting from $q_r^{(i,1)} := q^\iteridx_{\text{AGC},r}$, the iterates $q_r^{(i,l)}$ are for $l \in \mathbb{N}$ computed as
\begin{equation}
    \label{eq:TR_IRGNM_update}
    q_r^{(i,l+1)} := q_r^{(i,l)} + t_l d_r^{(i,l)} \in \qadriK,
\end{equation}
with step size $t_l \in [0,1]$. The directions $d_r^{(i,l)}$ are the solutions to \eqref{prob:lin_prob_rom} at $q_r^{(i,l)}$, using as regularization parameter $\alpha^{(i,l)} > 0$, such that \eqref{eq:choice_of_alpha_ROM} is satisfied. The trust region constraint is enforced using an Armijo-type backtracking line search. The value of $t_l$ is reduced iteratively by half until $q_r^{(i, l+1)} \in T^{(i)}$ and the decay condition~\eqref{eq:armijo_decay_condition} is satisfied. The initial step size is chosen as $\bar{t}_l = 1$ for $l = 1$, and as $\bar{t}_l = \min\{2t_{l-1}, 1\}$ otherwise.

\begin{figure}[b]
\centering
\includegraphics[width=0.80\textwidth]{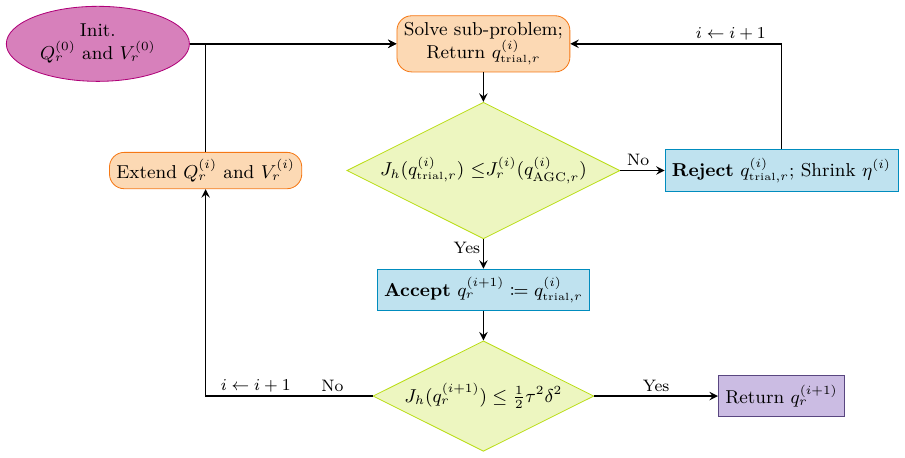}
\caption{\label{fig:flowchart} Simplified flowchart illustrating the main steps in TR-IRGNM.}
\end{figure}
The regularization parameter $\alpha^{(i,l)}$ is determined iteratively. Starting from an initial regularization parameter $\alpha^{(i,l)}_0$, we solve problem~\eqref{prob:lin_prob_rom} and check whether condition~\eqref{eq:choice_of_alpha_ROM} is satisfied. If the left-hand inequality in~\eqref{eq:choice_of_alpha_ROM} is violated, the parameter $\alpha^{(i,l)}$ is doubled and $d_r^{(i,l)}$ recomputed. If the right-hand inequality fails and $\alpha^{(i,l)} > 10^{-14}$, the parameter is halved and $d_r^{(i,l)}$ recomputed. This process is repeated until~\eqref{eq:choice_of_alpha_ROM} is satisfied, as outlined in Algorithm~\ref{algo:IRGNM}. 
For $l > 1$, we set ${\alpha^{(i,l)}_0 := \alpha^{(i,l-1)}}$, i.e., the most recently accepted regularization parameter, and for $l = 1$, we set $\alpha^{(i,1)}_0 \coloneqq \alpha^{(i-1,2)}$. The initial regularization parameter is defined as $\alpha^{(0,1)}_0 := \tilde{\alpha} > 0$.

The IRGNM terminates, and the current iterate $q_r^{(i,l)}$ is returned as trial solution to \eqref{eq:TR_Problem}, i.e., $\qtrial := q_r^{(i,l)}$, if one of the following conditions is met. First, the reduced discrepancy principle is satisfied:
\begin{equation}
    \label{eq:red_disc_princ}
    J^\iteridx_r(q_r^{(i,l)}) < \frac{1}{2} \tilde{\tau}^2 (\tilde{\delta}^\iteridx)^2,
\end{equation}
where $\tilde{\delta}^\iteridx \geq \delta$ denotes a possibly modified noise level, and $\tilde{\tau} > 1$. Alternatively, termination occurs if the iterate approaches the boundary of the trust region, that is,
\begin{equation}
    \label{eq:tr_bound_cond}
    \beta_1 \eta^\iteridx \leq \errest(q_r^{(i,l)}),
\end{equation}
which ensures reliable progress near the boundary and prevents inaccurate updates.

\vspace{\baselineskip}
\textbf{Acceptance of the trial step and modification of the TR radius:}
After computing the trial step $\qtrial$ using the reduced-order model, we must verify that the following sufficient decrease condition is satisfied:
\begin{equation}
\label{eq:EASDC}
J_h(\qtrial) \leq J_r^\iteridx(q_{\text{AGC},r}^\iteridx).
\end{equation}
This criterion, together with condition~\eqref{eq:AGC_decay_cond}, guarantees a decrease in the high-fidelity objective, i.e.,
\begin{equation*}
J_h(\qtrial) < J_h(q^\iteridx_r) \text{ for all } i \in \mathbb{I}.
\end{equation*}
However, verifying condition~\eqref{eq:EASDC} requires the expensive evaluation of the full-order objective $J_h$ at $\qtrial$, even though $\qtrial$ might be rejected afterwards. To address this issue, we first evaluate a cheap sufficient condition
\begin{equation}
\label{eq:suff_cond}
J_r^\iteridx(\qtrial) + \Delta^{J, \iteridx}(\qtrial) < J_r^\iteridx(q_{\text{AGC},r}^\iteridx)
\end{equation}
and a necessary condition
\begin{equation}
\label{eq:nec_cond}
J_r^\iteridx(\qtrial) - 
\Delta^{J, \iteridx}(\qtrial) \leq J_r^\iteridx(q_{\text{AGC},r}^\iteridx)
\end{equation}
to~\eqref{eq:EASDC}; see~\cite{YueMeerbergen2013, Keil2021nonco-54293}. Only, if this is inconclusive, we check \eqref{eq:EASDC} directly. If $\qtrial$ is accepted as the local minimizer $q_r^{(i+1)}$, the error tolerance $\eta^\iteridx$ is increased, provided that the reduced model accurately predicts the actual objective reduction. This is the case if
\begin{equation*}
\varrho^\iteridx \coloneqq \frac{J_h(q_r^\iteridx) - J_h(q_r^{(i+1)})}{J_r^\iteridx(q_r^\iteridx) - J_r^{(i+1)}(q_r^{(i+1)})} \geq \beta_2, \quad \text{with } \beta_2 \in \left[\nicefrac{3}{4}, 1\right).
\end{equation*}
In the case of rejection, the subproblem \eqref{eq:TR_Problem} is resolved, $\eta^\iteridx$ is reduced to $\beta_3 \eta^\iteridx$ for $\beta_3 \in (0,1)$.

An exception to this procedure arises when the AGC is returned as the trial parameter, i.e., $\qtrial = q^\iteridx_{\text{AGC},r}$. Then $q^\iteridx_{\text{AGC},r}$ is directly accepted as the next iterate, i.e. $q_r^{(i+1)} := q^\iteridx_{\text{AGC},r}$. The reduced spaces are then updated accordingly, and the trust region is contracted by setting $\eta^\iteridx := \beta_3 \eta^\iteridx$.
\begin{remark}
The IRGNM procedure described above does not guarantee that a trial solution $\qtrial$ satisfying~\eqref{eq:EASDC} will be found. The algorithm could, in principle, enter an infinite loop, leading to stagnation of the optimization process. However, we point out that such a case was never observed in the numerical experiments presented in Section~\ref{sec:num_exper}. Following, e.g., in~\cite{Keil2021nonco-54293}, we therefore assume that stagnation does not occur.
\end{remark}
\begin{algorithm2e}
\DontPrintSemicolon
\caption{TR-IRGNM}\label{algo:TR_IRGNM}
\small
\KwData{Noise level $\delta$, discrepancy parameter $\tau, \tilde{\tau} > 1$, initial guess $q^{(0)} \in \Q$, initial regularization $\tilde{\alpha}$, regularization center $q_\circ$, Armijo parameter $\kappa_\text{arm} > 0$, boundary parameter $\beta_1 \in (0,1)$, tolerance for enlargement of the radius $\beta_2 \in [3/4,1)$ and shrinking factor $\beta_3 \in (0,1)$.\;}
Set $k = 0$ and initialize the RB-ROM by constructing $Q_r^{(0)} \subset Q_h$ and $V_r^{(0)} \subset V_h$ as in \eqref{eq:Qbasis_init} and \eqref{eq:Vbasis_init}, respectively.\;
\While{$J_h(q_r^{(i,l)}) > \frac{1}{2}\tau^2 \delta^2$}{
   Compute $q^\iteridx_{\text{AGC},r}$ according \eqref{eq:armijo_decay_condition}. \label{algo:line:1}\;
	\If{$q_r^{(i,0)} \notin T^\iteridx$ or not \eqref{eq:AGC_decay_cond}}{Update $Q_r^\iteridx$ and $V_r^\iteridx$ as defined in Section~\ref{ssec:const_rb}, 
    and go to line~\ref{algo:line:1}.
    }
	\If{\eqref{eq:TR_condition}}{
		Set $\qtrial = q^\iteridx_{\text{AGC},r}$.\;
	} \Else{
		Solve \eqref{prob:prob_ROM} for $\qtrial$ as defined in Section~\ref{ssec:rb_TR_IRGNM}.
		
	} 
	\If{$\qtrial = q^\iteridx_{\text{AGC},r}$}{
		Accept $q_r^{(i+1)} = \qtrial$, update $Q_r^\iteridx$ and $V_r^\iteridx$ at $q_r^{(i+1)}$, as explained in Section~\ref{ssec:const_rb}.\;
		Shrink radius $\eta^{(i+1)} = \beta_3\eta^\iteridx$, and go to line~\ref{line:final_line}.
	}
	\If{\eqref{eq:suff_cond}}{
		Accept trial step as $q_r^{(i+1)} = \qtrial$, update $Q_r^\iteridx$ and $V_r^\iteridx$ at $q_r^{(i+1)}$, as explained in Section~\ref{ssec:const_rb}.\;
		\If{$\varrho^\iteridx > \beta_2$}{
				Enlarge radius $\eta^{(i+1)} = \beta_3^{-1}\eta^\iteridx$.\;
			}
	} \ElseIf{not \eqref{eq:nec_cond}}{
		Reject trial step, set $q_r^{(i+1)} = q_r^\iteridx$, keep $Q_r^\iteridx$ and $V_r^\iteridx$ and shrink radius $\eta^{(i+1)} = \beta_3\eta^\iteridx$.\;
	} \Else {
		\If{\eqref{eq:EASDC}}{
			Accept the trial step as $q_r^{(i+1)} = q^\iteridx_{\text{AGC},r}$, update $Q_r^\iteridx$ and $V_r^\iteridx$ at $q_r^{(i+1)}$, as explained in Section~\ref{ssec:const_rb}.\;
			\If{$\varrho^\iteridx > \beta_2$}{
				Enlarge radius $\eta^{(i+1)} = \beta_3^{-1}\eta^\iteridx$.\;
			}
		} \Else {
			Reject trial step, set $q_r^{(i+1)} = q_r^\iteridx$, keep $Q_r^\iteridx$ and $V_r^\iteridx$ and shrink radius $\eta^{(i+1)} = \beta_3\eta^\iteridx$.\;
		}
	}
	In case of acceptance update $\tilde{\delta}^{(i)}$, $\alpha_0^{(i)}$\;
	$i \leftarrow i + 1$. \label{line:final_line}\;	
}
\end{algorithm2e}
\subsection{Construction of the reduced spaces}
\label{ssec:const_rb}
An essential aspect of the TR-IRGNM is the construction of the reduced spaces $Q_r^\iteridx$ and $V_r^\iteridx$. The general goal thereby is to keep the dimensionality of these spaces as low as possible (for a fast computation) while preserving a high level of approximation accuracy. Various strategies exist for constructing suitable reduced spaces. In this work, we focus on \emph{adaptive enrichment}  startegies~\cite{keil2022adaptive, kartmann_adaptive_2023, QGVW17}, which extends the reduced bases by incorporating snapshots obtained during the optimization. Note that other basis update strategies also are compatible with Algorithm~\ref{algo:TR_IRGNM}, such as skipping enrichment steps or removing basis vectors that became obsolete; see, e.g.,~\cite{Banholzer2020AnAP, Banholzer2022Trust}.

\vspace{\baselineskip}
\textbf{Adaptive enrichment:}
The bases of $Q_r^\iteridx$ and $V_r^\iteridx$ are iteratively enriched using full-order solutions at $q_r^\iteridx$, resulting in spaces with increasing dimensions:
\begin{equation*}
Q_r^{(0)} \subset Q_r^{(1)} \subset \dots \subset Q_r^{(i_*(\delta,\yd))} \subset Q_h,
\quad
V_r^{(0)} \subset V_r^{(1)} \subset \dots \subset V_r^{(i_*(\delta,\yd))} \subset V_h.
\end{equation*}
Following the argumentation in Section~\ref{ssec:rb_TR_IRGNM} and the methodology introduced in~\cite{kartmann_adaptive_2023}, we aim to enrich the reduced parameter spaces by incorporating full-order gradients ${\nabla J_h(q_r^{(i)}) \in Q_h^K}$. These gradients are promising candidates for new basis vectors, as they represent directions of steepest descent of the objective functional $J_h$ at the reduced parameter $q_r^{(i)}$; cf.~Section~\ref{ssec:rb_TR_IRGNM}. However, as discussed in~\cite{kartmann_adaptive_2023}, alternative enrichment strategies may also be employed to improve local descent behavior. For instance, second-order information such as the Hessian products of $J_h$ in certain directions can be incorporated to guide basis enrichment more effectively. For the reduced state spaces, we add the full-order primal and adjoint trajectories at $q_r^{(i)}$, namely $u_h(q_r^{(i)}) \in V_h^K$ and $p(q_r^{(i)}) \in V_h^K$, as basis vectors. This is a standard approach in trust-region reduced basis (TR-RB) methods; see~\cite{QGVW17, Keil2021nonco-54293, keil2022adaptive, Banholzer2020AnAP, Banholzer2022Trust, klein2025multifidelitylearningreducedorder}. It ensures accurate local approximations of the full-order model within the reduced framework.

To avoid enriching the basis with vectors that contribute little additional information and unnecessarily increase the reduced space dimension, we apply Proper Orthogonal Decomposition (POD); cf. \cite{BBMV24,Pin08}. For the reduced parameter space (and analogously for the state space), basis vectors $B^\iteridx \subset Q_h$ are selected from a set of local snapshots $E^\iteridx \subset Q_h$, e.g., $E^\iteridx = \{\nabla J^1_h(q_r^{(i)}) , \dots, \nabla J^K_h(q_r^{(i)})\}$, such that they optimally represent the data $E^\iteridx$ up to a tolerance $\epsilon_\text{POD} > 0$ in a least-squares sense, i.e.,
\begin{equation*}
B^\iteridx = \argmin_{W \subset Q_h} \dim(\mathrm{span}(W)) \quad \text{s.t.} \quad \sum_{e \in E^\iteridx} \| e - \Pi_{\mathrm{span}(W)} e \|_{Q_h}^2 < \epsilon^2_\text{POD},
\end{equation*}
where $\Pi_{\mathrm{span}(W)}$ denotes the orthogonal projection onto $\mathrm{span}(W)$. The current basis is augmented by the selected vectors $B^\iteridx$ and subsequently re-orthogonalized. In practice, we employ hierarchical approximate POD (HaPOD)~\cite{himpe2018hierarchical}, a variant of POD that enables improved computational efficiency through parallelization.

Let $\Psi_Q^{(i)} \subset Q_h$ denote the basis of $Q_r^{(i)}$, i.e., $Q_r^{(i)} = \mathrm{span}( \Psi_Q^{(i)})$, and let similarly ${\Psi_V^{(i)} \subset V_h}$ denote the basis of $V_r^{(i)}$. The basis of the initial parameter space is defined as
\begin{equation}\label{eq:Qbasis_init}
\bm{\Psi}_Q^{(0)} := \texttt{HaPOD}\left[\left\{\bm{q}_{\circ,h},\, \bm{q}_r^{(0)},\, \bm{\nabla J}_h(q_r^{(0)})\right\}, \epsilon_\text{POD}\right].
\end{equation}
For $i \in \mathbb{I}$, after obtaining an iterate $q^{(i+1)} \coloneqq \qtrial$ satisfying \eqref{eq:EASDC}, the basis is updated via
\begin{equation*}
\bm{\Psi}_Q^{(i+1)} := \texttt{orthog.}\left[\bm{\Psi}_Q^{(i)} \cup \texttt{HaPOD}[\bm{E}^{(i)}, \epsilon_\text{POD}]\right],
\end{equation*}
where $\bm{E}^{(i)} := \nabla \bm{J}_h(q_r^{(i+1)}) \in \mathbb{R}^{N_Q \times K}$. The reduced state spaces $V_r^{(i)}$ are constructed analogously using full-order primal and adjoint states instead of gradients, that is,
\begin{equation*}
\bm{E}^{(i)} := \left\{\bm{u}_h(q_r^{(i+1)}),\, \bm{p}_h(q_r^{(i+1)})\right\} \in \mathbb{R}^{N_V \times 2K},
\end{equation*}
for $i\in \mathbb{I}$. The initial basis is given by
\begin{equation}\label{eq:Vbasis_init}
\bm{\Psi}_V^{(0)} := \texttt{HaPOD}\left[\left\{\bm{u}_h(q_r^{(0)}),\, \bm{p}_h(q_r^{(0)})\right\}, \epsilon_\text{POD}\right].
\end{equation}
If $q_r^{(i,0)} \notin T^\iteridx$ or $q_{\mathrm{AGC},r}^{(i)}$ does not satisfy \eqref{eq:AGC_decay_cond}, the reduced spaces are updated again using the same rule described above, but with a (iteratively) reduced $\epsilon_\text{POD}$. These rules uniquely define $Q_r^{(i)}$ and $V_r^{(i)}$ for all $i \in \mathbb{I}$.

\begin{remark}
If the parameter field is stationary in time, the POD is not required for the reduced parameter space. In this case, the parameter space is interpolated as for the elliptic problems studied in {\em\cite{kartmann_adaptive_2023}}, i.e.,
\begin{equation*}
\bm{\Psi}_Q^{(i+1)} := \texttt{orthog.}\left[\bm{\Psi}_Q^{(i)} \cup \bm{E}^\iteridx\right].
\end{equation*}
\end{remark}

\vspace{\baselineskip}
\subsection{Technical aspects}
\label{ssec:technical_aspects}
Before proceeding to the numerical experiments, we have to comment on the implementation of the error estimators and the projection operator, since both significantly affect the efficiency of the TR-IRGNM. 
%

\vspace{\baselineskip}
\textbf{Evaluation of the Error Estimator:}  
One of the most computationally expensive steps per evaluation in TR-IRGNM is the error estimation, as it requires computing the full-order residuals and their corresponding Riesz representatives for each time step and parameter value tested. A common strategy to mitigate these costs is to construct an orthonormal basis for the image space of the residuals $r^k_\text{pr}(\cdot\,; \cdot)$ and $r^k_\text{ad}(\cdot, \cdot\,; \cdot)$ and to project the residuals onto this basis; cf.~\cite{buhr2014numerically}.

While this significantly reduces the online evaluation cost, it incurs a high offline cost due to the need to reconstruct the image basis and reproject each time  the reduced spaces are adapted. This trade-off can negate the benefits of the reduced online cost, particularly in algorithms like TR-IRGNM, where the reduced spaces are iteratively refined. 
Since the error estimators are only used during the backtracking of $q^{(i,l)}_r + d^{(i,l)}_r$ into the trust region and are not required to compute $d^{(i,l)}_r$ by solving \eqref{prob:lin_prob_rom}, the offline cost outweighs the computational savings during parameter evaluation. Therefore, we instead compute the error estimates directly using the full-order residuals, avoiding the offline projection entirely.

\vspace{\baselineskip}
\textbf{Implementation of the projection operator:}
To implement the projection operator~\eqref{eq:projection_operator}, let assume for the moment that there exist bounds $0 < L < U < \infty$ such that the admissible set is given by
\begin{equation}
\label{eq:simple_bounded_Qad}
\Qad = \left\{ q \in \Q \,\middle|\, L \leq q \leq U \text{ a.e. in } (0,T) \times \Omega \right\}.
\end{equation}
Furthermore, assume that the space $Q_h$ is spanned by first-order Lagrange finite element basis functions. Under these assumptions hold $q_h \in \qadK$ if and only if $\bm{q}_h^k \in [L, U]^{N_Q}$ for all $k \in \mathbb{K}$. The discretized projection operator $P_{\qadK} : Q_h^K \rightarrow  \qadK$ can thus be implemented component-wise:
\begin{equation*}
\bm{P}_{\qadK}(q_h)^k_j \coloneqq 
\begin{cases}
L,& \text{if } \bm{q}^k_j \leq L,\\
\bm{q}^k_j,& \text{if } L < \bm{q}^k_j < U,\\
U,& \text{if } U \leq \bm{q}^k_j ,
\end{cases}
\end{equation*}
for $j \in \{ 1, \dots, N_Q\}$ and $k \in \mathbb{K}$. Although this implementation is straightforward and efficient for the FOM, it presents a drawback for the RB-ROM: a reduced parameter $q_r \in Q_r^K$ must be lifted to its full-order representation, which can be computationally expensive compared to evaluating the RB-ROM itself. To mitigate this problem, we introduce a cheap sufficient condition for verifying $q_r \in \qadrK$, stated in the following lemma.
\newpage
\begin{lemma}
Let $q \in \qadrK$, $d \in Q_r^K$, and let ${B \in \mathbb{R}^{N_Q \times n_Q}}$ be the basis matrix of $Q_r$, i.e., ${B = [\tilde{\bm{q}}_1, \dots, \tilde{\bm{q}}_{n_Q}]}$. Define for all $k \in \mathbb{K}$:
\begin{equation*}
\epsilon^k \coloneqq \min_{1 \leq i \leq N_Q} \left( \frac{U - \bm{q}_i^k}{\|B_{i,:}\|_{\mathbb{R}^{n_Q}}}, \frac{\bm{q}_i^k - L}{\|B_{i,:}\|_{\mathbb{R}^{n_Q}}} \right) \geq 0,
\end{equation*}
where $\bm{q}^k$ is the DoF-vector of $q^k$ with respect to the full-order basis of $Q_h$. If ${\|\bm{d}^k\|_{\mathbb{R}^{n_Q}} \leq \epsilon^k}$ holds for all $k \in \mathbb{K}$, we have $q + d \in \qadrK$.
\end{lemma}
\begin{proof}
For all $i \in \{1, \dots, N_Q\}$ and $k \in \mathbb{K}$, we have
\begin{equation*}
|[B \bm{d}^k]_i| \leq \max_{\|\delta^k\|_{\mathbb{R}^{n_Q}} \leq \epsilon^k} |[B \delta^k]_i| 
= \epsilon^k \|B_{i,:}\|_{\mathbb{R}^{n_Q}} 
\leq \min \left( U - \bm{q}_i^k, \bm{q}_i^k - L \right),
\end{equation*}
which implies that $q + d \in \qadrK$.
\end{proof}
This lemma can be used to compute the projection onto $\qadrK$ more efficiently, needed for solving~\eqref{prob:lin_prob_rom}, with projected gradient decent. Given a current iterate $q_r^{(i,l)} \in \qadrK$, the radii $\epsilon^k$ can be precomputed. Then, for a proposed update ${d_r \in (Q_r)^K}$, we check whether $\|\bm{d}^k\|_{\mathbb{R}^{n_Q}} \leq \epsilon^k$ for all $k \in \mathbb{K}$. In this case, no further computation is required, since $\bm{P}_{\qadrK}(q_r + d_r) = q_r + d_r.$ Otherwise, the update $q_r + d_r$ must be lifted to $Q^K_h$ and the projection carried out as explained before on the FOM level.
\section{Numerical experiments}
\label{sec:num_exper}
In this section, we evaluate the performance of the TR-IRGNM algorithm introduced in Section~\ref{sec:TR-IRGNM} using the benchmark setup described in Example~\ref{ex:default_example}. We consider four distinct scenarios, covering both stationary and time-dependent parameter fields for each case in Example~\ref{ex:default_example}. For each scenario, we apply the following algorithms:
\begin{enumerate}
    \item \textbf{FOM-IRGNM}: The standard IRGNM method, which employs the full-order model to solve the (linearized) primal and adjoint problems (cf. Algorithm~\ref{algo:IRGNM}).
        \item \textbf{TR-IRGNM}: The trust-region IRGNM (Algorithm~\ref{algo:TR_IRGNM}), tested with POD tolerances logarithmically spaced from $10^{-9}$ to $10^{-14}$.
\end{enumerate}
Both methods are implemented in \texttt{Python} using the \texttt{pyMOR} framework~\cite{doi:10.1137/15M1026614}, for model handling and reduction. The source code is available at \cite{source_code}. All experiments were conducted on the PALMA II HPC cluster at the University of Münster, funded by the DFG (INST 211/667-1), using compute nodes equipped with Intel Skylake Gold 6140 CPUs @ 2.30~GHz and 92~GB RAM.

\subsection{Implementation details}
For the numerical experiments, the time interval $[0,1]$ and the spatial domain ${\Omega := (0,1)^2}$ were chosen. The temporal discretization is performed with a time step of $\Delta t = \frac{1}{50}$ (that is, $K = 50$) and the spatial discretization of $Q, V,$ and $C$, using first-order Lagrange finite element (FE) basis functions in quadrilateral cells, resulting in $N_C = N_Q = N_V = 90,601$ degrees of freedom (DOFs). The noisy measurement was artificially constructed, by obtaining the exact solution $u_h^\mathsf e \in V_h^K$ for a reference parameter $q_h^\mathsf e \in Q_h^K$, adding uniformly distributed noise $\xi \in V_h^K \setminus \{0\}$ with noise level $\delta > 0$, defining
\begin{equation*}
y^\delta \coloneqq  \C_h u_h^\mathsf e + \delta \frac{\xi}{\|\xi\|_{V_h^K}}.
\end{equation*}
The noise level was set to $\delta = 10^{-5}$ and the parameters regarding the IRGNM are set to
\begin{align*}
    \theta = 0.4, \quad     
    \Theta = 1.95, \quad 
    \tau = \tilde{\tau} = 3.5, \quad 
    \tilde{\delta}^\iteridx  = \delta.
\end{align*}
The parameters controlling the trust region are chosen as follows: 
\begin{align*}
    \eta^{(0)} = 0.1, \quad
    \beta_1 = 0.95, \quad     
    \beta_2 = 0.75, \quad 
    \beta_3 = 0.5, \quad 
    \tilde{\alpha} = 10^{-5}.
\end{align*}

In all experiments, the constraints defining $\Q_\text{ad}$ were enforced using the projection operator described in Section~\ref{ssec:technical_aspects}, with $q_\text{a} \equiv L = 0.001$ as the lower bound and $U = 10^{3}$ as the upper bound. The reference value for the Tikhonov term was set as the background of the field of ground truth parameters (see below), that is, $q_\circ \equiv 3$. This value was also used as an initial guess for all runs, $q^{(0)} \equiv q_\circ$. The observation operator is the canonical embedding from $V$ to $H$, for which $\| \C\|_{\mathcal{L}(V, H)} = 1$ is valid.

The discretized linearized problems~\eqref{prob:lin_prob_FOM} and \eqref{prob:lin_prob_rom} were solved using a projected gradient descent method, the step size was determined by a Barzilai-Borwein-type line search~\cite{azmi2023nonmonotone}. The optimization procedure was terminated if the first-order optimality condition is satisfied, after a maximum of $10^4$ iterations, or if the objective value remained constant (up to machine precision, $10^{-16}$) for five consecutive iterations. For computing the AGC, we set $\kappa_\text{arm} = 10^{-12}$. The line search was terminated if conditions~\eqref{eq:armijo_decay_condition} and~\eqref{eq:TR_condition} were not satisfied within 100 iterations. 
%
\subsection{Stationary parameter fields}
\subsubsection{Run 1: Reaction}
\label{ssec:run1}
Consider Example \ref{ex:default_example} (i). We are looking to reconstruct a non time-varying reaction field ($\Q = L^2(\Omega)$), governing \eqref{eq:reaction_problem} for $f \equiv 1$. The coercivity constant $\alpha_{q_h}$ for $\bm{A}_h(q_h)$ is set to $1$ for all $q_h \in Q_h$. The exact parameter $q^\mathsf e$ is chosen as the sum of two Gaussian distributions and shifted by the background, i.e. $q^{\mathsf e}(\bm x) = q_\circ(\bm x) + q^{\mathsf e}_1(\bm x) + q^{\mathsf e}_2(\bm x)$ with 
\begin{align*}
q^{\mathsf e}_1(\bm{x}) & = \frac{1}{0.02 \pi } \exp \left( - \frac{1}{2} \left( \left( 
\frac{2x_1 - 0.5}{0.1}\right)^2 + \left( \frac{2x_2 - 0.5}{0.1}\right)^2
\right)\right), \\
q^{\mathsf e}_2(\bm{x}) & = \frac{1}{0.02 \pi } \exp \left( - \frac{1}{2} \left( \left( 
\frac{0.8x_1 - 0.5}{0.1}\right)^2 + \left( \frac{0.8x_2 - 0.5}{0.1}\right)^2
\right)\right).
\end{align*}
\begin{table}[ht]
\small
\setlength{\tabcolsep}{3.7pt}
\begin{tabularx}{\textwidth}{lc|cccccccc}
\toprule
	Algo. & $\epsPOD$ & $L^2$-rel. err. & $H^1$-rel. err. & time [s] & speed up & FOM solves & $n_Q$ & $n_V$ & o. iter \\
\midrule
	FOM & -- & -- & -- & 3396 & -- & 2306 & -- & -- & 13 \\
	TR & 1e-09 & 5.25e-02 & 1.63e-01 & 366 & 9.27 & 14 & 7 & 82 & 6 \\
	TR & 1e-10 & 5.24e-02 & 1.64e-01 & 347 & 9.77 & 14 & 7 & 85 & 6 \\
	TR & 1e-11 & 5.25e-02 & 1.64e-01 & 332 & 10.21 & 14 & 7 & 88 & 6 \\
	TR & 1e-12 & 5.25e-02 & 1.64e-01 & 323 & 10.50 & 14 & 7 & 88 & 6 \\
	TR & 1e-13 & 5.25e-02 & 1.64e-01 & 397 & 8.56 & 14 & 7 & 88 & 6 \\
	TR & 1e-14 & 5.24e-02 & 1.64e-01 & 371 & 9.14 & 14 & 7 & 91 & 6 \\
\bottomrule
\end{tabularx}
\caption{
Run 1: Comparison of FOM- and TR-IRGNM for different tolerances $\epsPOD$. Shown are the $L^2$- and $H^1$-relative errors of $q^{\text{\footnotesize TR}}$ w.r.t. $q^{\text{\footnotesize FOM}}$, total time, speed-up, number of FOM solves, final dimensions of $Q^\iteridx_r$, $V^\iteridx_r$, and outer iterations. The errors use the standard FE norms: $\|q^{\text{\footnotesize FOM}} - q^{\text{\footnotesize TR}}\|_{L^2(\Omega)} / \|q^{\text{\footnotesize FOM}}\|_{L^2(\Omega)}$ and analogously for $H^1$.
}
\label{tab:Run1_comparison}
\end{table}\\
A comparison of FOM- and TR-IRGNM for different POD tolerances $\epsPOD$ is presented in Table~\ref{tab:Run1_comparison}. The runs using TR-IRGNM converge for all tested $\epsPOD$, with a significantly reduced computational time, compared to FOM-IRGNM. The observed speed-up ranges from $8.5$ to $10.5$. Figure~\ref{fig:decays_run_1_and_2} (left) shows the decay of the objective $J_h(q_r^\iteridx)$ plotted against CPU time, for $\epsPOD = 10^{-12}$.
As shown in Table~\ref{tab:Run1_comparison}, TR-IRGNM achieves convergence with substantially fewer FOM evaluations, i.e., fewer solutions of the (linearized) primal and dual problems, compared to FOM-IRGNM. In TR-IRGNM, the FOM is evaluated only to compute $u_h(q_r^\iteridx)$ and $p_h(q_r^\iteridx)$, which serve as snapshots for enriching the reduced spaces $Q_r^\iteridx$ and $V_r^\iteridx$, and for evaluating the full-order discrepancy principle. However, it should be noted that evaluating the residuals in~\eqref{eq:primal_residual} and~\eqref{eq:adjoint_residual}, as well as performing the projection onto the admissible set, still incurs computational costs of order~$\mathcal{O}(N_V)$. All remaining computations are carried out within the reduced-order framework, thereby benefiting from reduced complexity.

\begin{figure}[b]
    \centering
    \includegraphics[width=\textwidth]{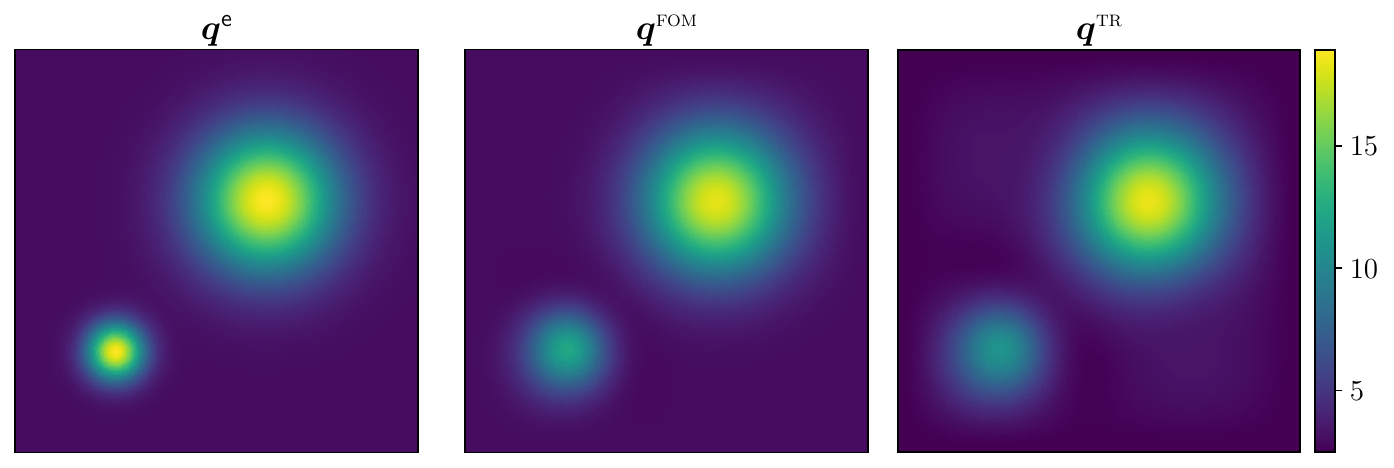} 
	\caption{Run 1: Exact parameter $\bm{q}^\mathsf{e}$ (left) and parameters reconstructed by FOM-IRGNM, $\bm{q}^\text{\footnotesize FOM}$ (middle) and by TR-IRGNM, $\bm{q}^\text{\footnotesize TR}$ (right), for $\epsPOD = 10^{-12}$.}
    \label{fig:run_1_reconst_param}
\end{figure}

\begin{figure}[t]
    \centering
    \includegraphics[width=0.9\textwidth]{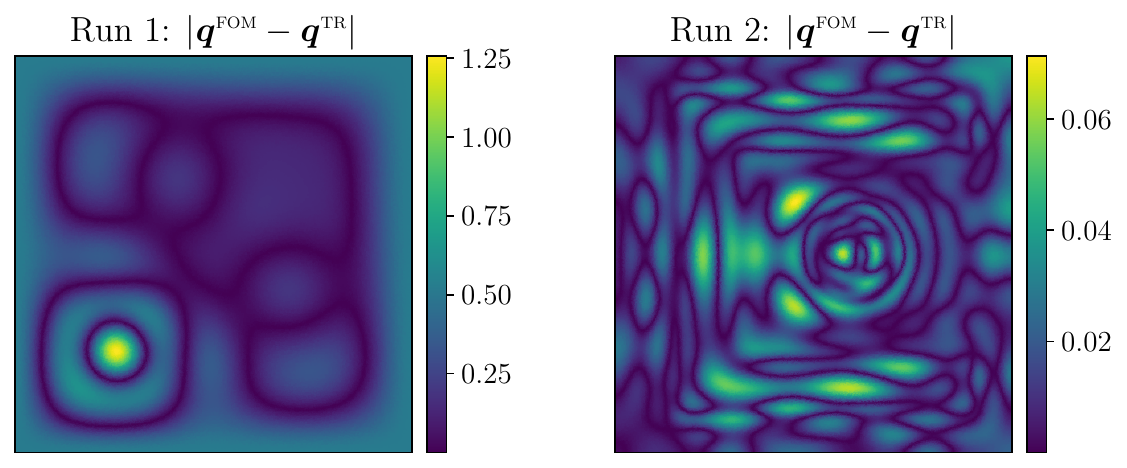} 
    \caption{Absolute difference of the reconstructed parameters obtained by FOM- and TR-TRGNM for Run 1 (left) and Run 2 (right), with $\epsPOD = 10^{-12}$.}
    \label{fig:diffs_run_1_and_2}
\end{figure}

Upon convergence, the dimensions of the (final) reduced spaces are similar across all tested values of~$\epsPOD$. The final dimension of the reduced state space $V_r^\iteridx$ lies between $82$ and $91$, while the reduced parameter space $Q_r^\iteridx$ consistently has a dimension of seven. Recall that at each enrichment step (except the first), only a single snapshot $\nabla_q J_h(q_r^\iteridx)$ is available to enrich $Q_r^\iteridx$, whereas 100 snapshots can be used for $V_r^\iteridx$. The final dimensions of $V_r^\iteridx$ indicate that a subset of the available state snapshots suffices to achieve local approximations within the all POD tolerances~$\epsPOD$ and that the number of available parameter snapshots is a limiting factor in the enrichment process. Increasing the number of parameter snapshots would enhance the quality of the local approximation.

\begin{figure}[b]
    \centering
    \includegraphics[width=0.95\textwidth]{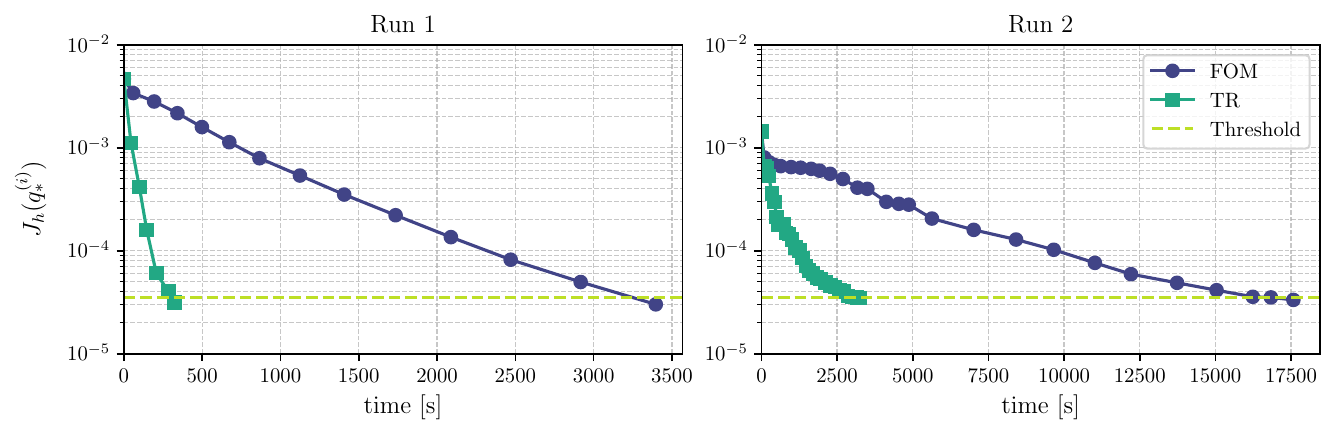} 
    \caption{
		Objective $J_h(q^\iteridx_\ast)$ plotted against the total computation time per outer iteration $i$, for Run 1 (left) and Run 2 (right). In the case of FOM-IRGNM, we set $q_\ast^\iteridx = q_h^\iteridx$ and for TR-IRGNM, $q_\ast^\iteridx = q_r^\iteridx$. The data shown is for the experiment with $\epsPOD = 10^{-12}$.}
    \label{fig:decays_run_1_and_2}
\end{figure}

Despite these limitations, the errors of the parameter fields reconstructed by TR-IRGNM, $q^\text{TR} \in Q_h$, relative to those from FOM-IRGNM, $q^\text{FOM} \in Q_h$, remain reasonably small: approximately $5\%$ in the $L^2(\Omega)$-norm and $16\%$ in the $H^1(\Omega)$-norm. Moreover, all reconstructions successfully capture the dominant features of the exact parameter $q\mathsf{e}$. The reconstructed parameters for $\epsPOD = 10^{-12}$ are shown in Figures~\ref{fig:run_1_reconst_param} and~\ref{fig:diffs_run_1_and_2}. The evolution of the approximate local solutions $q_r^\iteridx$ throughout the optimization process is depicted in Figure~\ref{fig:run1_evolution} in Appendix~\ref{sec:appendix}. There, we observe that, starting from the background parameter ($q_\circ = q^{(0)}$), the reaction coefficient gradually develops into the two distinct local peaks: first the larger peak $q^{\mathsf{e}}_1$, followed by the emergence of the smaller peak $q^{\mathsf{e}}_2$. Furthermore, comparing this to the basis vectors of $Q^\iteridx_r$ (Figure~\ref{fig:run1_q_basis}), we find, as expected, that the differences $q_r^{(i+1)} - q_r^\iteridx$ closely align with the corresponding basis vectors $\tilde{q}_j$ added during the updates.

\subsubsection{Run 2: Diffusion}
We are now considering the situation from Example \ref{ex:default_example} (ii), aiming to reconstruct the stationary diffusion field, for \eqref{eq:diffusion_problem} with $f \equiv 1$. Remark that in Example \ref{ex:default_example} the parameter is assumed to be in $H^2(\Omega)$. However, since we do not want to construct FE spaces in $H^2(\Omega)$, we  instead choose $\Q = H^1(\Omega)$ (cf. \cite{kartmann_adaptive_2023}). The coercivity constant is chosen as $\alpha_{q_h} := \essinf_{\Omega} q_h$. Analogue to \cite{RBL} and \cite{kartmann_adaptive_2023} is the exact parameter $q^\mathsf e$ given as $q^{\mathsf e}(\bm x) = q_\circ(\bm x) + 2 (\chi_{\Omega_1}(\bm x)- \chi_{\Omega_2}(\bm x))$, where  
\begin{align*}
\Omega_1 & = [5/30, 9/30] \times [3/30, 27/30] \\
& \quad\cup \left([9/30, 27/30] \times \left([3/30, 7/30] \cup [23/30, 27/30] \right) \right),\\
\Omega_2 & = \left\{ x \in \Omega\,|\,\|x - (18 / 30, 15/30)^T\|_{\mathbb{R}^2} < 4 / 30\right\}.
\end{align*}
\begin{table}[b]
\small
\setlength{\tabcolsep}{3.7pt}
\begin{tabularx}{\textwidth}{lc|cccccccc}
\toprule
	Algo. & $\epsPOD$ & $L^2$-rel. err. & $H^1$-rel. err. & time [s] & speed up & FOM solves & $n_Q$ & $n_V$ & o. iter \\
\midrule
	FOM & -- & -- & -- & 17577 & -- & 13828 & -- & -- & 26 \\
	TR & 1e-09 & 5.97e-03 & 5.15e-02 & 3141 & 5.60 & 58 & 28 & 295 & 28 \\
	TR & 1e-10 & 5.81e-03 & 4.88e-02 & 3084 & 5.70 & 60 & 29 & 323 & 29 \\
	TR & 1e-11 & 5.34e-03 & 4.92e-02 & 3539 & 4.97 & 60 & 29 & 358 & 29 \\
	TR & 1e-12 & 5.49e-03 & 4.79e-02 & 3253 & 5.40 & 60 & 29 & 359 & 29 \\
	TR & 1e-13 & 6.39e-03 & 5.54e-02 & 3333 & 5.27 & 58 & 28 & 347 & 28 \\
	TR & 1e-14 & 6.29e-03 & 5.37e-02 & 3360 & 5.23 & 58 & 28 & 361 & 28 \\
\bottomrule
\end{tabularx}
\caption{Run 2: Comparison of FOM- and TR-IRGNM for different tolerances $\epsPOD$, analogue to Table~\ref{tab:Run1_comparison}.}
\label{tab:Run2_comparison}
\end{table}
Table~\ref{tab:Run2_comparison} shows key values for FOM- and TR-IRGNM, analogous to Table~\ref{tab:Run1_comparison}. Many of the observations made for Run~1 also apply here. However, the speed-ups achieved by TR-IRGNM are less pronounced, ranging from $4.97$ to $5.70$. As before, the limited number of available parameter snapshots is a constraining factor, necessitating more frequent updates (up to 29) of the reduced spaces. This negatively impacts the overall speed-up in several ways.

First, at each iteration, both reduced spaces must be updated, which involves computationally costly tasks: evaluating the FOM, performing POD, and re-orthogonalizing. Second, inadequate local approximations require more frequent tuning of the parameter $\alpha^{(i,l)}$ to satisfy condition~\eqref{eq:choice_of_alpha_ROM}, leading to additional overhead and often resulting in weaker regularization. Third, the frequent enrichments increase the dimensions of the reduced spaces, making the evaluation of the reduced-order models more expensive.

Regarding this last point, it is important to note that the reduced spaces are enriched iteratively. As a result, their dimensions may over time grow larger than necessary to provide accurate (local) approximations at the current iterate. This occurs because the basis may retain vectors corresponding to earlier local solutions that no longer contribute significantly. Consequently, the enrichment strategy can lead to an inflated reduced basis by accumulating (locally) obsolete basis vectors.
\begin{figure}[t]
    \centering
    \includegraphics[width=\textwidth]{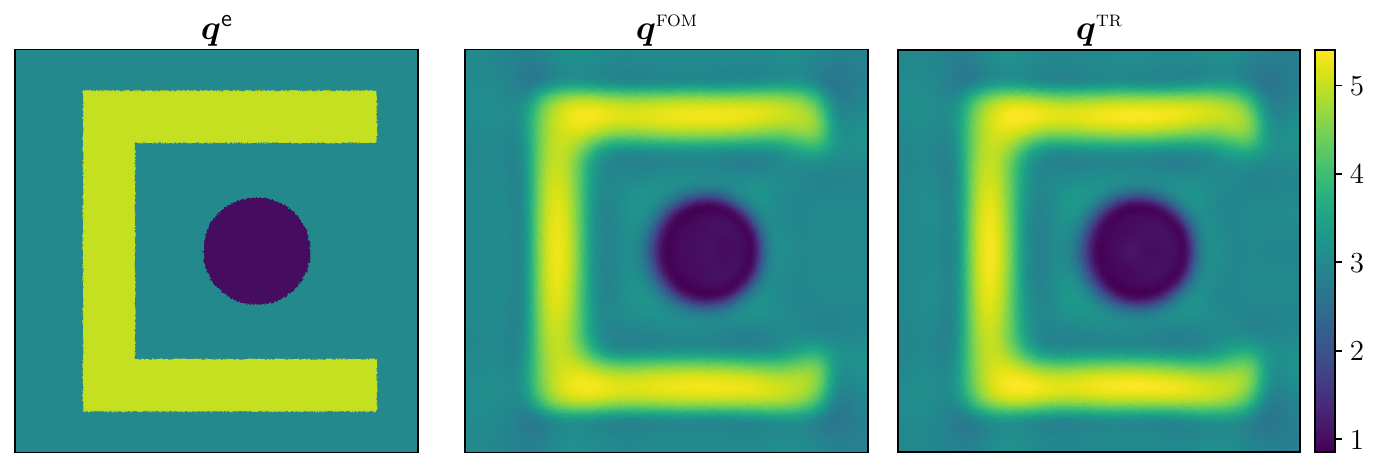} 
	\caption{
		Run 2: Exact parameter $\bm{q}^\mathsf{e}$ (left) and parameters reconstructed by FOM-IRGNM, $\bm{q}^\text{\footnotesize FOM}$ (middle) and by TR-IRGNM, $\bm{q}^\text{\footnotesize TR}$ (right), for $\epsPOD = 10^{-12}$.}
    \label{fig:run_2_reconst_param}
\end{figure}
Furthermore, we observe that the reconstructed parameters exhibit smaller relative errors compared to those in Run~1. This improvement results from the larger reduced parameter basis constructed in this run, which enhances the representational capacity of the reduced model. This effect becomes even more significant in problems involving time-dependent parameter fields (see below). Despite this improvement, for all reconstructed diffusion fields, including those obtained with FOM-IRGNM, we see that the discontinuities present in the exact parameter $q^\mathsf{e}$ are noticeably smoothed out in the reconstructions. 
\subsection{Time-varying parameter fields}
\subsubsection{Run 3: Reaction}
Let now consider the situation from Run 1 again (Example \ref{ex:default_example} (i)), where the parameter field varies in time; $\Q = L^2(0,1; L^2(\Omega))$. Analogue to Run 1, the (time-dependent) coercivity constants $\alpha_{q^k_h}$ can for all $k \in \mathbb{K}$ and $q_h \in Q_h^K$ be set to $1$. The exact parameter $q^\mathsf e$ consists of the background $q_\circ$ and a time-varying part, given by two Gaussian distributions from Run 1 where the peak-size is controlled by a sinus function, i.e., 
\begin{equation*}
q^{\mathsf e}(t, \bm x) = q_\circ(\bm x) + \sin (\pi t) \left(q^{\mathsf e}_1(\bm x) + q^{\mathsf e}_2(\bm x) \right).
\end{equation*}
\begin{table}[ht]
\small
\setlength{\tabcolsep}{3.7pt}
\begin{tabularx}{\textwidth}{lc|cccccccc}
\toprule
	Algo. & $\epsilon_{\text{POD}}$ & $L^2$-rel. err. & $H^1$-rel. err. & time [s] & speed up & FOM solves & $n_Q$ & $n_V$ & o. iter \\
\midrule
	FOM & -- & -- & -- & 4922 & -- & 2496 & -- & -- & 12 \\
	TR & 1e-09 & 2.71e-02 & 1.26e-01 & 1013 & 4.86 & 8 & 10 & 49 & 3 \\
	TR & 1e-10 & 2.94e-02 & 1.30e-01 & 900 & 5.47 & 8 & 17 & 60 & 3 \\
	TR & 1e-11 & 1.11e-02 & 5.51e-02 & 1055 & 4.66 & 8 & 20 & 56 & 3 \\
	TR & 1e-12 & 7.90e-03 & 4.07e-02 & 933 & 5.27 & 8 & 28 & 51 & 3 \\
	TR & 1e-13 & 2.18e-03 & 1.32e-02 & 944 & 5.21 & 8 & 35 & 59 & 3 \\
	TR & 1e-14 & 4.44e-03 & 2.25e-02 & 900 & 5.47 & 8 & 102 & 60 & 3 \\
\bottomrule
\end{tabularx}
\caption{
Run 3: Comparison of FOM- and TR-IRGNM for different tolerances $\epsPOD$, analogous to Table~\ref{tab:Run1_comparison}. The relative errors of the parameter field $q^\text{TR} \in Q^K_h$ with respect to $q^\text{FOM} \in Q^K_h$, computed as $\|q^{\text{\footnotesize FOM}} - q^{\text{\footnotesize TR}}\|_{L^2(0,T;L^2(\Omega))} / \|q^{\text{\footnotesize FOM}}\|_{L^2(0,T;L^2(\Omega))}$ and analogously for $H^1$, using standard FE norms.}
\label{tab:Run3_comparison}
\end{table}
As shown in Table~\ref{tab:Run3_comparison}, in contrast to the case with time-invariant parameter fields, a clear dependence of both the final dimension of $Q^\iteridx_r$ and the relative errors on the POD tolerance $\epsPOD$ is observed. Specifically, decreasing the tolerance leads to smaller relative errors and a higher final dimension of the reduced parameter space. A direct comparison reveals that the relative errors are consistently smaller than those observed in Run~1. Notably, for $\epsPOD = 10^{-9}$ are the relative errors already smaller than those in Run~1 for $\epsPOD = 10^{-14}$. Moreover, while the number of outer iterations for FOM-IRGNM remains comparable to that of Run~1 ($12$ vs.\ $13$), it is halved in the runs using TR-IRGNM ($3$ vs.\ $6$). This behavior can be attributed to the fact that snapshot availability for the parameter space is no longer a limiting factor. The gradients $\nabla J_h(q^\iteridx_r) \in V^K_h$ are time dependent, and thus each enrichment provides $K = 50$ snapshots, resulting in more expressive reduced parameter spaces compared to Run~1. As a consequence, is the number of required updates reduced, and the quality of the approximations in $Q^\iteridx_r$ is improved. The dimension of the reduced state space remains relatively stable (between $49$ and $60$), exhibiting behavior consistent with Runs~1 and~2.

\begin{figure}[b]
    \centering
    \includegraphics[width=\textwidth]{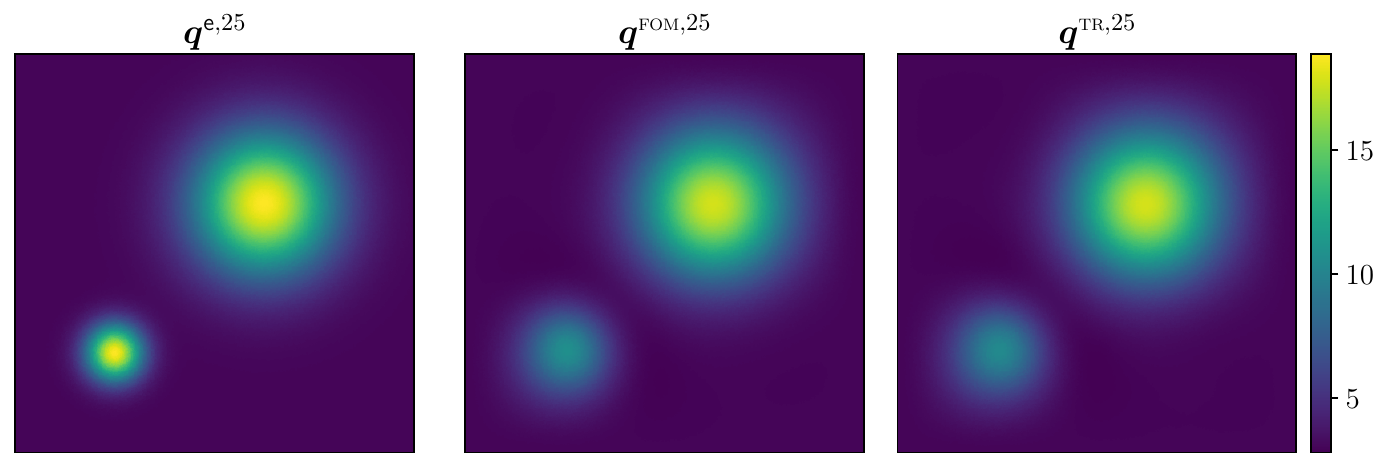} 
	\caption{Run 3: Exact parameter $\bm{q}^{\mathsf{e},25}$ (left) and parameters reconstructed by FOM-IRGNM, $\bm{q}^{\text{\footnotesize FOM},25}$ (middle) and by TR-IRGNM, $\bm{q}^{\text{\footnotesize TR},25}$ (right), for $\epsPOD = 10^{-12}$, at time step $k = 25$.}
    \label{fig:run_3_reconst_param}
\end{figure}

\begin{figure}[t]
    \centering
    \includegraphics[width=0.9\textwidth]{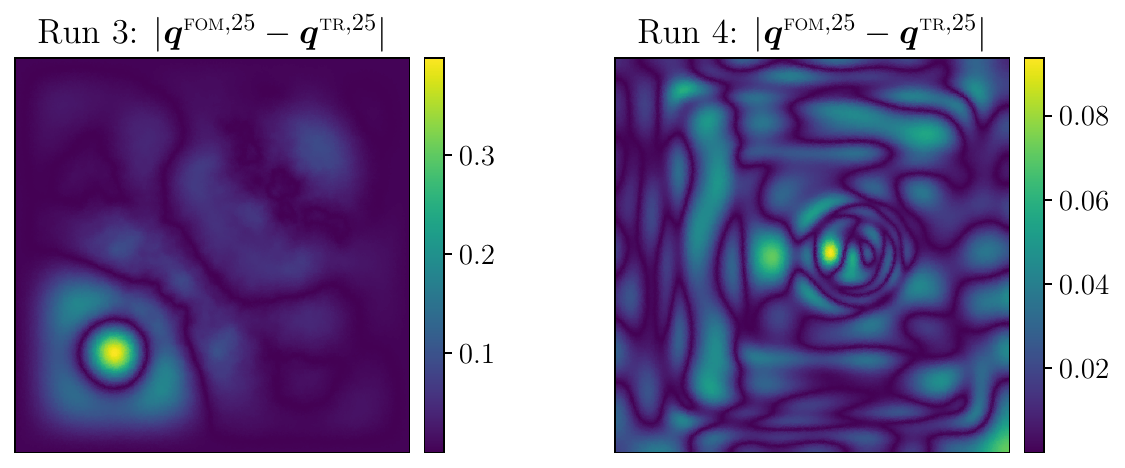} 
    \caption{Absolute difference of the reconstructed parameters obtained by FOM- and TR-TRGNM for Run 3 (left) and Run 4 (right), with $\epsPOD = 10^{-12}$}
    \label{fig:diffs_run_3_and_4}
\end{figure}

Despite the reduced number of outer iterations, is the overall speed-up achieved in Run~3 less pronounced than in Run~1. This is primarily due to the significantly higher computational cost of evaluating the full-order state and adjoint solutions for $q_h \in Q_h^K$ in the enrichment of the reduced spaces. In the time-dependent setting, this requires reassembling and inverting the matrix $\bm{A}_h(q_h^k)$ at every time step, which is substantially more expensive than the corresponding operations for the time-independent case of Run~1. In contrast, solving the reduced problem~\eqref{prob:lin_prob_rom} does not incur a proportionally higher cost. This is because the system matrices involved there are cached, when the linearized problems~\eqref{eq:state_prob_ROM} and~\eqref{eq:adj_prob_ROM} are solved for updated directions but fixed parameters $q^{(i,l)}_r$. Nevertheless, the additional computational savings from the reduced-order model are not sufficient to offset the added expense during the enrichment. Additionally, for the same reason, becomes the evaluation of the error estimator more computationally demanding, which further diminishes the overall speed advantage of the TR-IRGNM in this setting.

\begin{figure}[b] 
    \centering
    \includegraphics[width=0.95\textwidth]{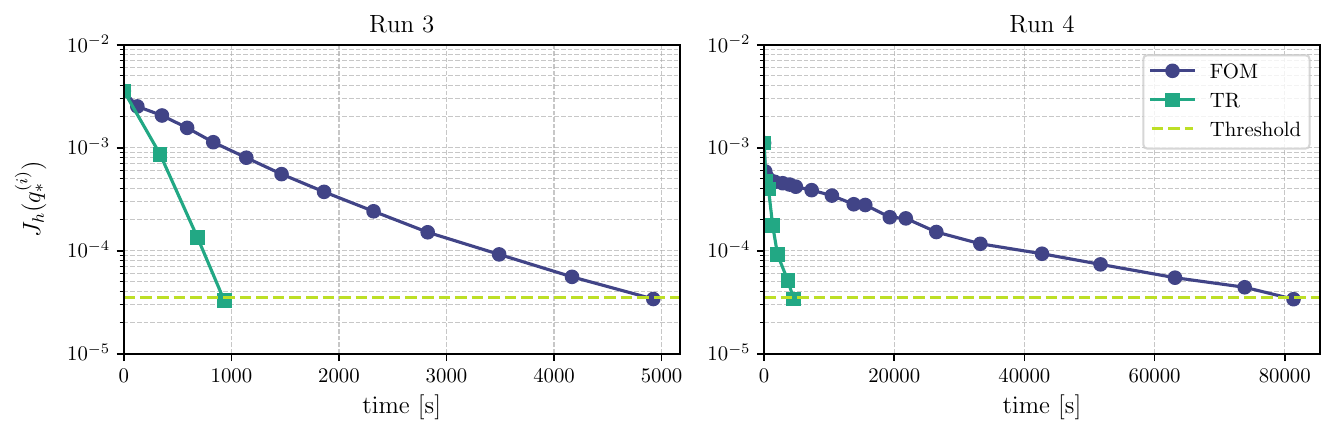} 
    \caption{
		Objective $J_h(q^\iteridx_\ast)$ plotted against the total computation, for Run 3 (left) and Run 4 (right). In the case of FOM-IRGNM, we set $q_\ast^\iteridx = q_h^\iteridx$ and for TR-IRGNM, $q_\ast^\iteridx = q_r^\iteridx$. The data shown is for the experiment with $\epsPOD = 10^{-12}$.
    }
    \label{fig:decays_run_3_and_4}
\end{figure}

\subsubsection{Run 4: Diffusion}
Similar to Run 3, we are looking to reconstruct the time-varying diffusion field in Example \ref{ex:default_example} (ii) for $f \equiv 1$. Let therefore, as in Run 2, $\Q = L^2(0,1; H^1(\Omega))$ and the coercivity constants $\alpha_{q^k_h}$ be set as $\alpha_{q^k_h} := \essinf_{\Omega} q^k_h$ for all $k \in \mathbb{K}$ and $q_h \in Q_h^k$. The exact parameter is constructed as in Run 3, keeping the background stationary and changing the remaining part, i.e. 
\begin{equation*}
q^{\mathsf e}(t, \bm x) = q_\circ(\bm x) + 2 \sin(\pi t) (\chi_{\Omega_1}(\bm x) - \chi_{\Omega_2}(\bm x)).
\end{equation*}
\begin{table}[ht]
\small
\setlength{\tabcolsep}{3.7pt}
\begin{tabularx}{\textwidth}{lc|cccccccc}
\toprule
	Algo. & $\epsilon_{\text{POD}}$ & $L^2$-rel. err. & $H^1$-rel. err. & time [s] & speed up & FOM solves & $n_Q$ & $n_V$ & o. iter \\
\midrule
	FOM & -- & -- & -- & 81348 & -- & 43695 & -- & -- & 21 \\
	TR & 1e-09 & 1.24e-02 & 1.12e-01 & 15765 & 5.16 & 32 & 28 & 492 & 15 \\
	TR & 1e-10 & 1.26e-02 & 1.11e-01 & 10384 & 7.83 & 24 & 40 & 392 & 11 \\
	TR & 1e-11 & 9.80e-03 & 8.82e-02 & 5381 & 15.12 & 16 & 53 & 234 & 7 \\
	TR & 1e-12 & 4.30e-03 & 4.30e-02 & 4548 & 17.89 & 14 & 68 & 197 & 6 \\
	TR & 1e-13 & 2.36e-03 & 2.77e-02 & 4901 & 16.60 & 14 & 94 & 220 & 6 \\
	TR & 1e-14 & 1.36e-03 & 1.52e-02 & 5047 & 16.12 & 14 & 158 & 215 & 6 \\
\bottomrule
\end{tabularx}
\caption{Run 4: Comparison of FOM- and TR-IRGNM for different tolerances $\epsPOD$, analogue to Table~\ref{tab:Run3_comparison}.}
\label{tab:Run4_comparison}
\end{table}
As in Run~3, Table~\ref{tab:Run4_comparison} reveals a clear dependency of the dimension of the reduced parameter space $Q_r^\iteridx$ and the relative reconstruction errors on the POD tolerance~$\epsPOD$. In addition, we observe a dependency of the final dimension of the reduced state space $V_r^\iteridx$, the achieved speed-up, and the number of outer iterations on~$\epsPOD$ as well. Notably, the number of outer iterations decreases as the POD tolerance becomes more restrictive. At the same time, the dimension of $V_r^\iteridx$ is significantly reduced with decreasing~$\epsPOD$, dropping from $492$ to $197$ up to $\epsPOD = 10^{-12}$, but shows a slight increase for even smaller tolerances. This behaviour can be attributed to the limited accuracy of local approximations for higher POD tolerances. In such cases, smaller trust regions or the inability to identify a sufficiently accurate step size $\alpha$ lead to more frequent enrichments of the reduced basis. Consequently, the reduced state space accumulates more snapshots, resulting in a larger overall dimension. These phenomena also explain the relatively modest speed-ups observed for high POD tolerances. 

As the tolerance becomes more restrictive (i.e., $\epsPOD \leq 10^{-11}$), the speed-ups increase and even exceed those observed in Run~2. As in Run~3, the time-dependent setting increases the computational cost of both basis enrichment and error estimation. However, in contrast to Run~2, the cost of the full-order IRGNM becomes also significantly higher in this setting. The primary reason is that solving the linearized problems~\eqref{prob:lin_prob_rom} requires considerably more iterations to converge than in Run~2, explaining the increased speed-ups.

\begin{figure}[ht] 
    \centering
    \includegraphics[width=\textwidth]{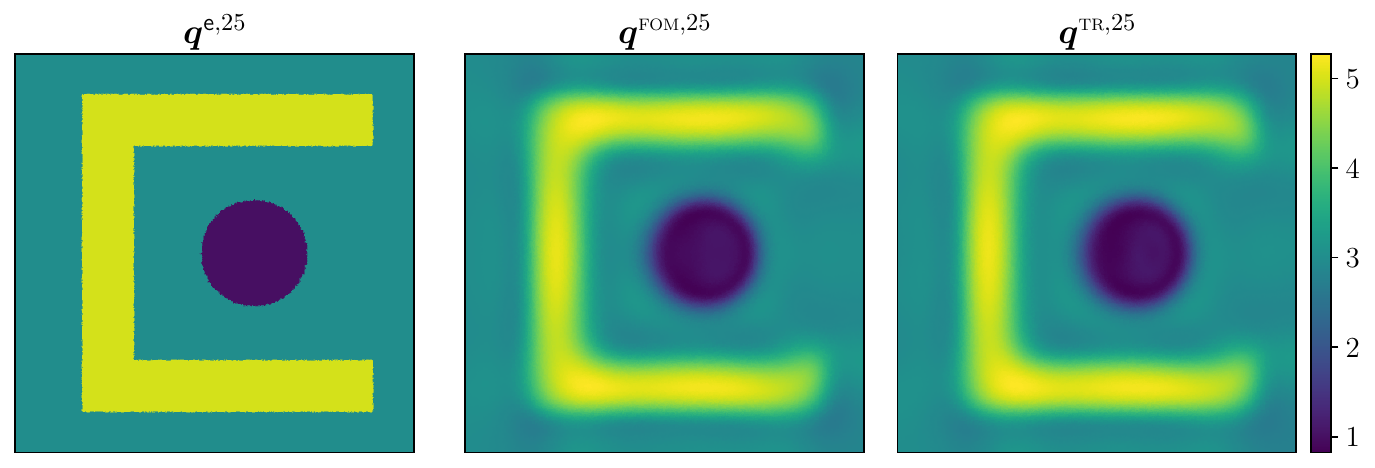} 
	\caption{Run 4: Exact parameter $\bm{q}^{\mathsf{e}, 25}$,  (left) and parameters reconstructed by FOM-IRGNM, $\bm{q}^{\text{\footnotesize FOM}, 25}$ (middle) and by TR-IRGNM, $\bm{q}^{\text{\footnotesize TR}, 25}$ (right), for $\epsPOD = 10^{-12}$, at time step $k = 25$.}
    \label{fig:run_4_reconst_param}
\end{figure}

\section{Conclusion \& Outlook}
We have presented an extension to the algorithm introduced in~\cite{kartmann_adaptive_2023}, which employs adaptively constructed reduced parameter and state spaces within an IRGNM optimization framework. This extension enables the application to parameter identification problems constrained by parabolic PDEs. The proposed modifications incorporate safeguards ensuring that the reduced spaces are enriched only by the most relevant modes. With this approach a balance between local approximation accuracy and the computational efficiency of the reduced-order models is maintained.

Numerical experiments demonstrate the effectiveness of the proposed algorithm in reconstructing reaction and diffusion fields. These experiments are conducted using parameter spaces with dimensionality equal to that of the finite element discretization of the state space, and they show that the algorithm achieves a substantial reduction in computational time while still maintaining the desired accuracy.

Subject of future research will be the investigation of stability in view of $\delta \to 0$. To this end the IRGNM stopping index, the FE discretization and the MOR dimension have to be coupled to the noise level $\delta$ in an appropriate way to prove convergence for $\delta \to 0$. 



\section*{Statements and declarations}
\bmhead{Funding} Benedikt Klein and Mario Ohlberger acknowledge funding from the Deutsche Forschungsgemeinschaft (DFG, German Research Foundation) under Germany's Excellence Strategy – EXC 2044/1 – 390685587, Mathematics Münster: Dynamics–Geometry–Structure.

\bmhead{Competing interests} The authors have no relevant financial or non-financial interests to disclose.
\bibliography{biblio}

\clearpage
\appendix
\section{Additional figures for Run 1}
\label{sec:appendix}

\begin{figure}[ht]
    \centering
    \includegraphics[width=\textwidth]{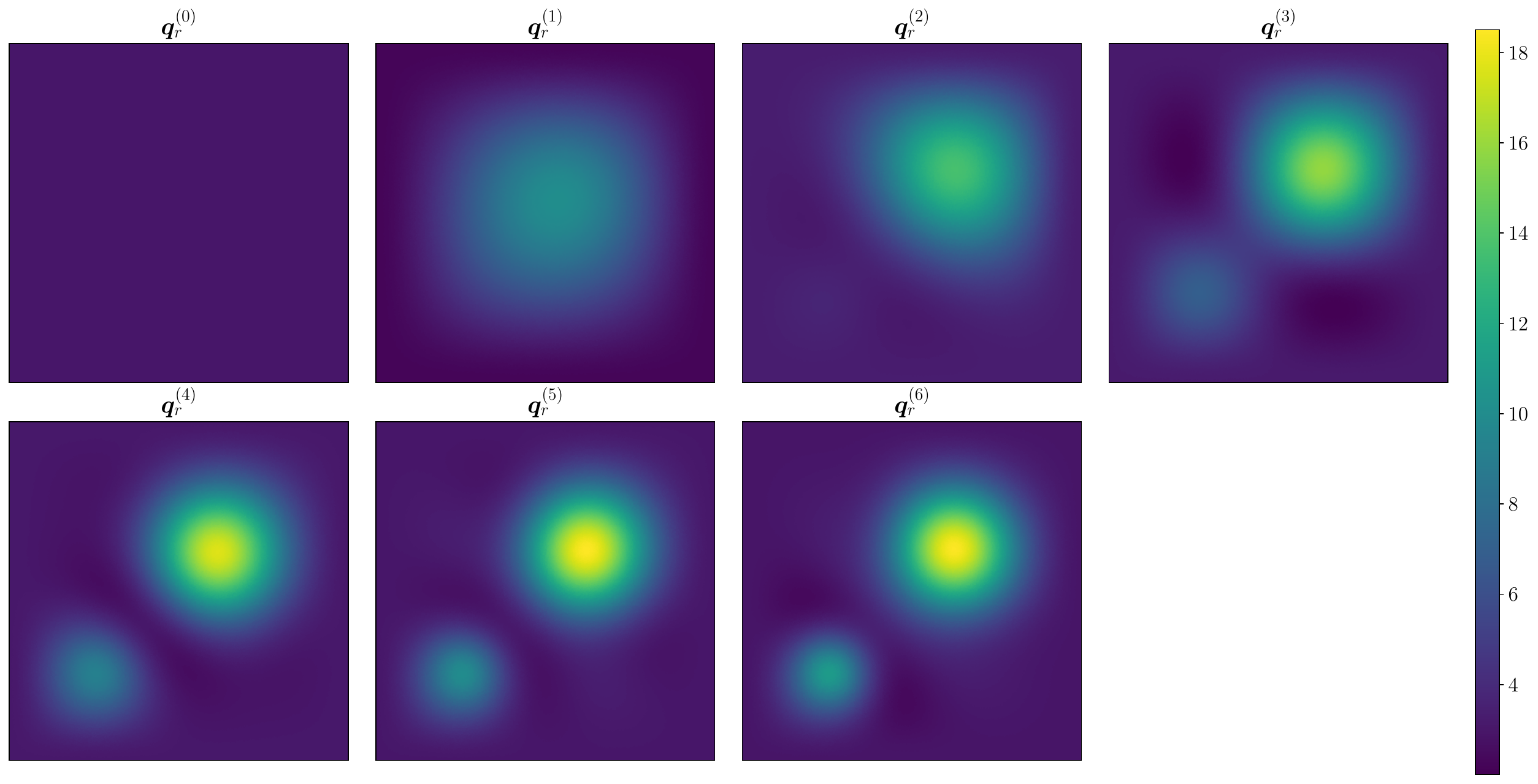} 
	\caption{Run 1: Approximative solutions $\bm{q}^\iteridx_r$ to the subproblems, obtained by TR-IRGNM for ${\epsPOD = 10^{-12}}$.}
    \label{fig:run1_evolution}
\end{figure}

\begin{figure}[ht] 
    \centering
    \includegraphics[width=\textwidth]{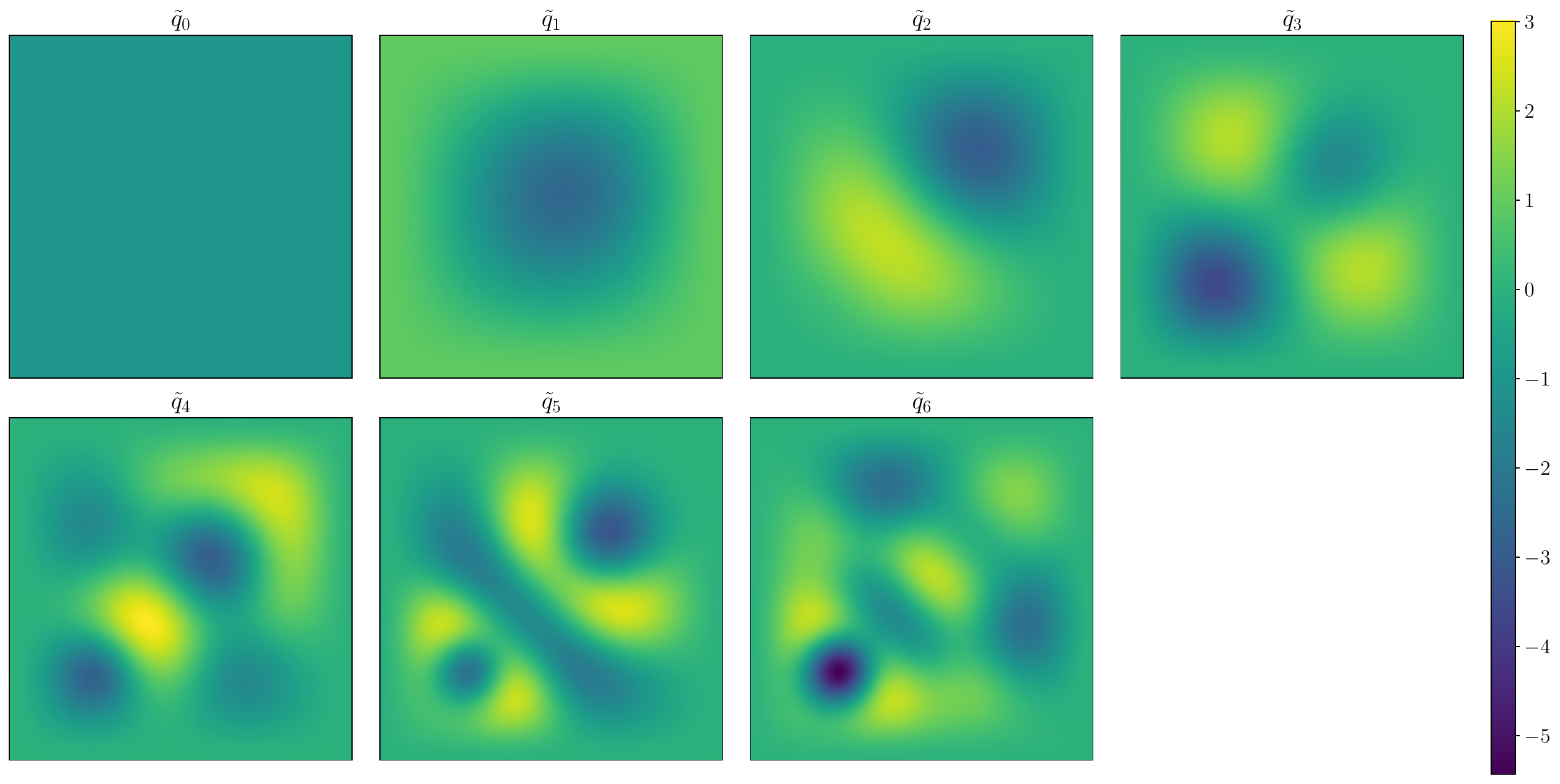} 
	\caption{Run 1: Basis vectors of $Q_r^{(6)}$ iteratively constructed by TR-IRGNM for $\epsPOD = 10^{-12}$.}
    \label{fig:run1_q_basis}
\end{figure}

\end{document}